\documentclass[10pt]{article}
\usepackage{amsfonts}
\usepackage{amsmath,amsthm,amssymb,latexsym,amscd,xcolor}
\usepackage{epsf}
\usepackage[pdftex]{graphicx}
\usepackage{tikz}
\usepackage{hyperref} 
\usepackage{appendix}
\usepackage{verbatim}
\newtheorem{theorem}{Theorem}[section]
\newtheorem{lemma}[theorem]{Lemma}
\newtheorem{corollary}[theorem]{Corollary}

\theoremstyle{definition}
\newtheorem{definition}[theorem]{Definition}

\theoremstyle{remark}
\newtheorem{remark}[theorem]{Remark}

\numberwithin{equation}{section}

\title{\bf Bloch Spectra for High Contrast Elastic Media}

\author{Robert Lipton\thanks{Department of Mathematics, Louisiana State University,
Baton Rouge, LA 70803, USA,
https://orcid.org/0000-0002-1382-3204, 
{\tt lipton@lsu.edu}}
\and 
Ruchira Perera, 
\thanks {*Corresponding Author, Department of Mathematics, Lockett Hall 343,  
Louisiana State University,
Baton Rouge, LA 70803, USA, https://orcid.org/0000-0001-9621-7359, 
{\tt jperer3@lsu.edu}}}

\date{}
\begin{document}
\maketitle
\begin{abstract}

Analytic representation formulas and power series are developed to describe the band structure inside periodic elastic crystals made from high contrast inclusions. We use source free modes associated with structural spectra to represent the solution operator of the Lam\'e system inside phononic crystals. Convergent power series for the Bloch wave spectrum are obtained using the representation formulas. An explicit bound on the convergence radius is given through the structural spectra of the inclusion array and the Dirichlet spectra of the inclusions. 
Sufficient conditions for the separation of spectral branches of the dispersion relation for any fixed quasi-momentum are identified. 
A condition is found that is sufficient for the emergence of band gaps.

\begin{flushleft}
\end{flushleft}

\end{abstract}

%
%
\maketitle

\section{Introduction}
\label{introduction}

 \setlength\itemsep{2em}
 
High contrast periodic elastic crystals  have been studied both theoretically and experimentally and have been shown to exhibit unique dispersive properties. One can split high contrast crystals into two classes based on the length scale of the crystal structure relative to the wavelength. When crystal geometry is on the same length scale as the elastic wave the dispersion is due to Bragg scattering and the patterned material is referred to as a phononic crystal, \cite{EconomouSigalas2,Siglas,Kushawas,Sigmund,Vassur,AmmariKangLee1,Li2}. Alternatively if the wave length lies above the crystal period a sub-wavelength resonance can be induced and this becomes the principal effect that controls wave dispersion. Crystals of this type are referred to as phononic metamaterials, \cite{AvilaGriso,Rohan, Smyshlyaev,Comi,Vondrejc}. 
In this article we address the former ``multiple scattering,'' problem when the wave length is on the scale of the heterogeneities. We consider periodic arrays with low wave velocity inclusions embedded in a high wave velocity medium (often referred to as the matrix). Such crystals exhibit novel dispersion  and are known to exhibit band gaps \cite{EconomouSigalas}. 



In this article new rigorous and explicit analytic representation formulas and power series are developed to describe wave dispersion inside phononic crystals. These results apply to wave propagation inside phononic crystals made from high contrast inclusions. The explicit formulas are used to investigate the propagation band structure of the crystal as a function of the inclusion geometry. The phononic elastic crystal is a composite of two materials each with different density and elasticity.  The propagation of a Bloch wave $h(x)$ at frequency $\omega$ inside the elastic crystal  is described by the differential equation,
\begin{equation}
\label{I1}
    -\nabla\cdot (\mathbf{C}(x)\mathcal{E} h(x))=\omega^2\rho(x)h(x),\; x\in\mathbb{R}^d,\; d=2,3.
    \end{equation}
Here $\mathbf{C}(x), x\in\mathbb{R}^d$, is the fourth rank tensor that represents the local elastic constants of the material and $\rho(x)$ is the local density. The crystal is taken to be infinite in extent and with out loss of generality the unit period cell is the cube $Y=(0,1]^d$. 
The Bloch wave $h(x)$ inside the crystal satisfies the $\alpha$ quasi-periodicity condition $h(x+p)=h(x)e^{i\alpha\cdot p}$, where $\alpha$ is the quasi-momentum in the first Brillouin zone $Y^*=(-\pi, \pi]^d$.  The piece wise constant elastic tensor and density are periodic and satisfy $\mathbf{C}(x)=\mathbf{C}(x+p)$ and $\rho(x)=\rho(x+p) $ with $p\in\mathbb{Z}^d, d=2,3$.  The crystal is composed of a periodic array of isolated inclusions $D$ surrounded by a second phase. The array of inclusions is described by the set $\Omega=\cup_{m\in\mathbb{Z}^d}(D+m)$, and the connected phase is described by $\mathbb{R}^d\setminus\Omega$. In this treatment the boundary of the inclusion is  taken to be $C^\infty$ smooth. 
The low velocity inclusions are embedded in a high velocity matrix, i.e., $\rho^1>\rho^2$ and $k>1$.
The piece wise  constant density and the piece wise constant elasticity tensor for the medium are written    
\begin{equation}
\label{Elastic tensor}
\begin{aligned}
\mathbf{C}(x) & =\mathbf{C}^1\chi_{\Omega}(x)+\mathbf{C}^2(1-\chi_\Omega(x))\\
\rho(x) & = \rho^1\chi_{\Omega}(x)+\rho^2(1-\chi_\Omega(x))
\end{aligned}
\end{equation}
where $\rho^1$, $\rho^2$ are constant densities and the elasticity tensor $\mathbf{C}^1:=C_{ijkl}:=\lambda_1\delta_{ij}\delta_{kl}+\mu_1(\delta_{ik}\delta_{jl}+\delta_{il}\delta_{jk})$  is isotropic and specified by Lam\'e constants $(\lambda_1,\mu_1)$ and $\mathbf{C}^1$  satisfies 
\begin{equation}\label{ellip}
    0<\gamma |\zeta|^2\leq\mathbf{C}^1\zeta:\zeta\leq \beta|\zeta|^2,
\end{equation}
for all $\zeta\in Sym^d
=\{\zeta\in\mathbb{R}^{d\times d}:\zeta=\zeta^T\}$. The elastic moduli describing
$\mathbf{C}^2$ are given by $\lambda_2=k\lambda_1$ and $\mu_2=k\mu_1$, $1\leq  k <\infty$, where $k$ represents the contrast between the two elastic materials. The symmetric gradient of the elastic displacement $u$ is denoted by $\mathcal{E}(u)$, given by,
\[
\mathcal{E}(u)=\frac{1}{2}(\nabla u+\nabla u^t),
\]
where the superscript $t$ denotes the matrix transpose. 
The corresponding co-normal derivative on $\partial\Omega$ is 
\begin{equation}
\label{I4}
 \partial_{n}u:=(\mathbf{C}^1\mathcal{E}(u))n
\end{equation}
where $n$ is the outward unit normal vector to $\partial\Omega$ and the Lam\'e operator $\mathcal{L}$ on $\mathbb{R}^d$, $d=2,3$ is defined to be
\begin{equation}
    \label{Lame Operator}
 \mathcal{L}u:= \nabla\cdot\mathbf{C}^1\mathcal{E}(u)=\mu_1\Delta u+(\lambda_1+\mu_1)\nabla(\nabla\cdot u).
\end{equation}
Here we prove the results for the $d=3$ case. Our approach also applies to the $d=2$ case, however the specifics differ and this will be reported in a separate publication.

In this paper we investigate the band structure as a function of the elastic contrast $k$ between the two materials, inclusion shape and placement inside the period cell. It is known that frequency band gaps open up for elastic crystals for sufficiently high contrast, see \cite{EconomouSigalas}.
For each $\alpha\in Y^*$ the Bloch eigenvalues $\omega^2$ are of finite multiplicity and denoted by $\xi_j(k,\alpha), \; j\in\mathbb{N}.$ We develop explicit series expansions in the contrast $k$ for each branch of the dispersion relation
\begin{equation}
\label{dispersion relation}
 \xi_j(k,\alpha)=\omega^2, \; j\in\mathbb{N}   
\end{equation}
that are valid for $k$ in a neighborhood of infinity.  The radii of convergence and convergence rate for the series are found to depend explicitly on the inclusion shape and placement within the period cell, see sections \ref{Radius of conv} and \ref{Example}.  Conditions sufficient for the separation of spectral branches of the dispersion relation for any fixed quasi-momentum are found, see section \ref{Radius of conv}.  We characterize the high contrast limit of the Bloch spectra and give sufficient conditions for the emergence of band gaps, see section \ref{sec-bandgap}. When the inclusion is symmetric a new spectral interlacing  property is found that is identical to that seen for scalar problems in acoustics, see \cite{HempleLennau}. The approach taken here is distinct from other approaches and as noted earlier is not asymptotic, instead it uses shape and configurational information contained in the structural spectra of the periodic array of inclusions. The structural spectra is identified here for the elastic problem and is a family of eigenvalues $\{\tau_i(\alpha)\}_{i=1}^\infty$, $\alpha\in Y^\ast$ associated with eigenvalue problems 
that encode the geometry of the crystal, see Definition \ref{structural}. For fixed $\alpha \in Y^\ast$ the eigenvalues $\{\tau_i(\alpha)\}_{i=1}^\infty$, are referred to as the quasi-periodic spectra of the crystal, see \eqref{specq} of section \ref{Hilbert setting}. We identify this spectrum with the spectrum of the 
well known Neumann Poincar\'e operator \cite{Khavinson}, \cite{AndoKangMiyanishi}, \cite{AndoYong} constructed in the quasi periodic setting, see Lemma \ref{Ht2.3}.

\begin{figure}
\centering
\begin{tikzpicture}
  \shade[yslant=-0.5,right color=gray!10, left color=black!50]
    (0,0) rectangle +(3,3);
  \shade[yslant=0.5,right color=gray!50,left color=gray!10]
    (3,-3) rectangle +(3,3);
  \shade[yslant=0.5,xslant=-1,bottom color=gray!10,
    top color=black!80] (6,3) rectangle +(-3,-3);
  \shade[ball color = orange!80, opacity = 0.4] (3.0,1.5) circle (1.5cm);
  \node [below] at (2.5,-0.25) {$Y\setminus D$};
  \node [right] at (3.5,1.0) {$D$};
\end{tikzpicture}
\caption{{\bf Inclusion geometry inside a period Cell.}}
 \label{plane}
\end{figure}
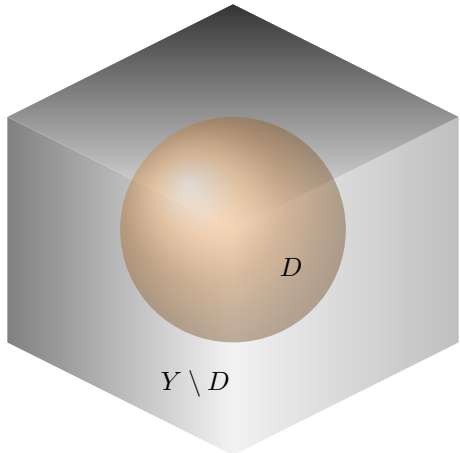

To proceed we complexify the problem and consider $k\in\mathbb{C}$. Now $\mathbf{C}(x)$ takes on complex values inside $Y\setminus D$ and divergence form operator is no longer uniformly elliptic. Our approach develops an explicit representation formula for $-\nabla\cdot(\mathbf{C}(x)\mathcal{E}(u))$ that holds for complex values of $k$. We identify the subset $z=\frac{1}{k}\in\Omega_0$ of $\mathbb{C}$ where this operator is invertible. The explicit formula shows that the solution operator $(-\nabla\cdot (\mathbf{C}^1\chi_D+\mathbf{C}^2\chi_{Y\setminus D}))\mathcal{E})^{-1}$ may be regarded more generally as a meromorhic operator valued function of $z$ for $z\in\Omega_0=\mathbb{C}\setminus S^\alpha,$ see section \ref{Representation} and Lemma \ref{PL1} . Here the set $S^\alpha$ consists of poles lying on the negative real axis with three accumulation points bounded away from $z=0$. These poles are in one to one relation to with the structural spectra $\{\tau_i(\alpha)\}_{i=1}^\infty$, $\alpha\in Y^\ast$. The interval of the negative real axis containing $S^\alpha$ can be bounded uniformly for $\alpha\in Y^\ast$ for a wide class of inclusion shapes and placements inside the unit period cell, see Sections \ref{Separation of spectra} and \ref{Example}. For the problem treated here we expand about $z=0$ and the set $S^\alpha$ is used to bound the radius of convergence for the power series. The spectral representation for $-\nabla\cdot (\mathbf{C}^1\chi_D+\mathbf{C}^2\chi_{Y\setminus D}))\mathcal{E}$ follows from the existence of a complete set of orthonormal set of quasi-periodic functions associated with the quasi-periodic resonances of the crystal, i.e., quasi periodic functions $v$ and real eigenvalues $\xi$ for which
\begin{equation}
\label{reduced eigval problem}
   -\nabla\cdot (\mathbf{C}^1\chi_D)\mathcal{E}(v)=-\xi \nabla\cdot \mathbf{C}^1\mathcal{E}(v).
\end{equation}
These resonances are shown to be connected to the spectra of elastostatic Neumann-Poincar\'e operators associated with quasi periodic double layer potentials.  For $\alpha=0$ and for a single sphere in $\mathbb{R}^3$ these correspond to the elastostatic eigenvalues identified in \cite{Youjun}. Both elastostatic Neumann-Poincar\'e (N-P) operators and associated elastosatic resonances have been the focus of theoretical investigations \cite{AndoKangMiyanishi}. Unlike scalar problems \cite{Khavinson}, these investigations have shown that N-P operator is not compact even for smooth domains. Instead the seminal work \cite{AndoKangMiyanishi} shows that the N-P operator is polynomially compact. These results have been applied in analysis of cloaking by anomalous localized resonance for the elastostatic system \cite{AndoYong}. The explicit spectral representation for the operator $(-\nabla\cdot (\mathbf{C}^1\chi_D+\mathbf{C}^2\chi_{Y\setminus D})\mathcal{E})$ developed here is crucial for elucidating the interaction between the contrast $k$ and the quasi-periodic resonances of the crystal, see \eqref{B2}, \eqref{H35}, and \eqref{H36}. The spectral representation is applied to analytically continue the band structure $\xi_j(k,\alpha)=\omega^2,\;j\in\mathbb{N},\; \alpha\in Y^*$ for $k$ onto $\mathbb{C},$ see Theorem \ref{BT1}. Application of the contour integral formula for spectral projections \cite{TKato1}, \cite{TKato2} and \cite{TKato3} delivers an analytic representation formula for the band structure, see section \ref{Representation}. We apply perturbation theory in section \ref{Representation} together with a calculation provided in section \ref{sec-derivation} to find explicit formula for the radii of convergence for the power series $\xi_j(k,\alpha)$ about $1/k=0$. The formula shows that the radius of convergence and separation between different branches of the dispersion relation are determined by: 1) the distance of the origin to the nearest pole $z^*$ of $(-\nabla\cdot (\mathbf{C}^1\chi_D+\mathbf{C}^2\chi_{Y\setminus D}))\mathcal{E})^{-1},$ and 2) the separation between distinct eigenvalues in the $z=1/k\rightarrow 0$ limit, see Theorems \ref{Rt1} and \ref{Rt2}. These theorems provide conditions on the contrast guaranteeing the separation of spectral bands that depend explicitly upon $z^*, j\in\mathbb{N}$ and $\alpha\in Y^*.$ Error estimates for series truncated after $N$ terms follows directly from the formulation.

Next we apply these results and develop bounds on the convergence radii for a wide class of inclusions called buffered geometries. A buffered geometry is described by a randomly placed inclusion inside the unit period cell with a finite distance of separation between inclusion and cell boundary, see section \ref{Separation of spectra}. For these geometries we demonstrate that the poles of $(-\nabla\cdot (\mathbf{C}^1\chi_D+\mathbf{C}^2\chi_{Y\setminus D}))\mathcal{E})^{-1}$ associated with the quasi-periodic spectra are bounded away from the origin uniformly for $\alpha\in Y^*$. The quasi-periodic spectra $\{\tau_i(\alpha)\}_{i\in\mathbb{N}}$ associated with a buffered geometry is shown to lie inside the interval $-1/2<\tau^-\leq\tau_i(\alpha) \leq 1/2,$ for every $\alpha\in Y^*,$ see Theorem \ref{Gt1} and Corollary \ref{Gc2}. The lower bound $\tau^-$ is independent of $\alpha\in Y^*$ and depends explicitly on the geometry of the inclusions. This control insures that the associated poles of $(-\nabla\cdot (\mathbf{C}^1\chi_D+\mathbf{C}^2\chi_{Y\setminus D}))\mathcal{E})^{-1}$ are uniformly bounded away from the origin and provides an explicit nonzero radius of convergence for the power series representation for the band structure $\xi_j(k,\alpha)=\omega^2$ for each $j\in \mathbb{N}$ and $\alpha\in Y^*$, see Theorems \ref{Rt1} and \ref{Rt2}. In section \ref{Example} we apply these observations to periodic assemblages of buffered spheres. Here a buffered sphere is characterized by a period containing a randomly placed sphere within the interior of the unit call. The term buffer referrers to the distance between the boundary of the sphere to the boundary of the cell. For this case we recover explicit formulas for the radii of convergence of the power series expansion for $\xi_j(k,\alpha)$ and explicit conditions for the separation of spectral bands in terms of the distance between sphere boundary and cell boundary. It is important to emphasize that the results on separation of spectra and convergence of power series are not asymptotic results but are valid for an explicitly delineated regime of finite contrast.

Earlier work on effective properties for periodic and random media \cite{Dell'Antonio},  \cite{KantorBergman}, \cite{Milton}, show that the effective elasticity for a composite medium is an analytic function of the contrast. The effective elasticity function is seen to be nonzero and analytic off the negative real axis and is determined by its singularities and zeros. Estimates for effective properties are obtained from partial knowledge of the singularities and zeros. The work \cite{Bruno} develops power series solutions to bound the poles and zeros of the effective elasticity function. This provides bounds on the effective elasticity function for the class of inclusion geometries  discussed here. Asymptotic expansions for Bloch eigenvalues are developed and applied to the high contrast setting for two dimensional elasticity in \cite{AmmariKangLee1}. The expansions are in terms of the contrast and developed using a boundary integral perturbation approach based on the generalized Rouch\'e's theorem. In that work the high contrast band structure is identified and a criterion for band gap opening is given in 2 dimensions. The criterion sufficient for band gap opening in three dimensions given here is consistent with the one presented in \cite{AmmariKangLee}. However in the present context it is shown to follow from the Lipschitz continuity of Bloch eigenvalues with respect to quasi-momentum at fixed contrast. Recently the appropriate structural spectrum for the Helmholtz operator has been identified and used to quantitatively capture the photonic band structure for wave propagation problems in high contrast media for TE electromagnetic modes in \cite{bibentry1}. This knowledge is used to establish  explicit formulas for both photonic  pass band and band gap frequency intervals as functions of the inclusion geometry in \cite{LiptonViator2}.

Last we point out that earlier related work using different methods provide explicit power series representations for the spectra of two dimensional photonic  metamaterials.  The work of \cite{Fortes1} develops a convergent power series representation for the spectra of metamaterial crystals made from high contrast frequency dependent rods. Explicit of radii of convergence are obtained. The work of \cite{Fortes2} demonstrates the existence of convergent power series expansions for the spectrum of metamaterials containing high contrast positive or negative dielectric inclusions. It is shown that the power series converge for sufficiently small contrast. The work of \cite{chenlipton} provides a power series representation of spectra for metamaterials made from periodic configurations containing both high contrast dielectric rods and frequency dependent dielectric rods. The existence of traveling waves with phase velocity opposite to the direction of the Poynting vector is rigorously shown to follow from Maxwell's equations.

The paper is organized as follows: In the next section we introduce the Hilbert space formulation of the problem and the variational formulation of the quasi-static resonance problem. The completeness of the eigenfunctions associated with the quasi-static spectrum is established and a spectral representation for the operator $(-\nabla\cdot (\mathbf{C}^1\chi_D+\mathbf{C}^2\chi_{Y\setminus D}))\mathcal{E})$ is obtained. These results are collected and used to continue the frequency band structure into the complex plane, see Theorem \ref{BT1}
section \ref{Band structure}. Spectral perturbation theory \cite{TKato3} is applied to recover the power series expansion for Bloch spectra in section \ref{representation}. The leading order spectral theory is developed for quasi-periodic $\alpha\neq 0$ and periodic $\alpha=0$ problems in sections \ref{Spectrum-quasiperiodic} and \ref{Spectrum-periodic}. The main theorems on radius of convergence and convergence rates are given by Theorems \ref{Rt1}, \ref{Rt2}, and \ref{Rt4} are presented in section \ref{Radius of conv}. The class of buffered inclusions is introduced in section \ref{Separation of spectra} and the explicit radii of convergence for a random suspension of disks is presented in section \ref{Example}. The structure of  high contrast limit spectra is given in section \ref{sec-bandgap}. Explicit formulas for each term of the power series expansion is recovered and expressed in terms of layer potentials in section \ref{sec-layerpotential}. In section \ref{sec-explicit first order} the explicit formula for the first order correction in the power series is presented in the form of the Dirichlet energy of the solution of a transmission boundary value problem. This formula follows from the layer potential representation for the first term and consistent with the first order correction obtained in the work of \cite{AmmariKangLee1} for two dimensions. The explicit formulas for the convergence radii are derived in section \ref{sec-derivation} as well as hands on proofs of Theorems \ref{Rt1}, \ref{Rt2} and the error estimates for the series approximation.


\section{Hilbert space setting, quasi-periodic resonances and representation formulas}
\label{Hilbert setting}
We denote the space of all $\alpha$ quasi-periodic complex vector valued functions belonging to $L^2_{loc}(\mathbb{R}^3)^3$ by  $L^2_\#(\alpha,Y)^3$ and the inner product is denoted by
\begin{equation}
\label{H1}
(u,v)=\int_{Y} u\cdot\overline{v}\;dx.
\end{equation}
For $\alpha\neq0$ the eigenfunctions $h$ for \eqref{I1} belong to the space
\begin{equation}
\label{H2}
 H^1_{\#}(\alpha, Y)^3 =\{h\in {H^1_{loc}}(\mathbb{R}^3)^3:h\; \text{is}\; \alpha\; \text{quasi-periodic}\}.
\end{equation}
This space does not contain the space of rigid motions hence the kernel of the symmetric gradient is zero and
the space $H^1_{\#}(\alpha, Y)^3$ is a Hilbert space under the inner product
\begin{equation}
\label{H3}
\langle u,v\rangle =\int_{Y}\mathbf{C}^1\mathcal{E}( u):\overline{\mathcal{E} (v)}\;dx.
\end{equation}

The periodic eigenfunctions of $\eqref{I1}$ associated with nonzero eigenvalues belong to the space
\begin{equation}
\label{H4}
 H^1_{\#}(0, Y)^3 =\left\{h\in {H^1_{loc}}(\mathbb{R}^3)^3:h\; \text{is}\; \text{periodic},\;\int_Y \rho\,h\,dx=0\right \}.   
\end{equation}
Note here that $H^1_{\#}(0, Y)^3$ does not contain the space of rigid motions so it is also a Hilbert space with the inner product $\langle u,v\rangle$ defined by \eqref{H3}.

In what follows we write $\rho=\rho(x)$ and
for any $k\in \mathbb{C}$, the weak formulation of the eigenvalue problem \eqref{I1} for $h$ and $\omega^2$ is given by 
\begin{equation}
    \label{H5}
B_{k}(u,v)=\omega^2(\rho u,v)\quad \text{for all}\quad v\in  H^1_{\#}(\alpha, Y)^3
\end{equation}
where $B_{k}:H^1_{\#}(\alpha, Y)^3\times H^1_{\#}(\alpha, Y)^3\mapsto\mathbb{C} $ is the sesquilinear form given by
\begin{equation}
\label{H6}
\begin{aligned}
B_{k}(u,v)&=\int_{Y}{\mathbf{C}(x)\mathcal{E}( u):\overline{\mathcal{E}( v)}\; dx}\\
&=k\int_{Y\setminus D}\mathbf{C}^1\mathcal{E}( u):\overline{\mathcal{E}(v)}\; dx+\int_{ D}\mathbf{C}^1\mathcal{E}(u):\overline{\mathcal{E}(v)}\; dx.
\end{aligned}
\end{equation}
Let $T_k^{\alpha}:H^1_{\#}(\alpha, Y)^3\mapsto H^1_{\#}(\alpha, Y)^3$ be the associated linear operator such that 
\begin{equation}
\label{H7}
\langle T_k^{\alpha}u,v\rangle=B_k(u,v)\; \text{for all}\; v\in H^1_{\#}(\alpha, Y)^3.
\end{equation}
Hence from \eqref{H5} the eigenvalue problem becomes finding the pair $\omega^2,u$ such that $\omega^2>0$ and  $u\in H^1_{\#}(\alpha, Y)^3$ for which
\begin{equation}
    \label{weakspectrum}
\langle T_k^{\alpha}u,v\rangle=\omega^2(\rho u,v).
\end{equation}
Now we identify the operator associated with this eigenvalue problem.
Let $F(v):H^1_{\#}(\alpha, Y)^3\mapsto\mathbb{C}$ be the linear functional such that $F(v)=(\rho u,v)$ for fixed $u\in H^1_{\#}(\alpha, Y)^3$. Then by the Riesz Representation Theorem, there is a unique $z_{\rho u}\in H^1_{\#}(\alpha, Y)^3$ such that
\begin{equation}\label{weaksolution}
\langle z_{u\rho },v\rangle=F(v)=(\rho u,v)\; \text{for all}\; v\in H^1_{\#}(\alpha, Y)^3.
\end{equation}
Let $-\mathcal{L}_\alpha$ be the Lam\'e operator associated with the bilinear form $\langle\cdot,\cdot\rangle$ defined on $H^1_\#(\alpha,Y)^3.$
so
\begin{equation*}\label{rep}
z_{\rho u}=-\mathcal{L}_\alpha^{-1}\rho u,
\end{equation*}
and
\[
\langle T_k^{\alpha}u,v\rangle=\omega^2\langle z_{\rho u},v\rangle=\omega^2\langle -\mathcal{L}_\alpha^{-1}\rho u,v\rangle\; \text{for all}\; v\in H^1_{\#}(\alpha, Y)^3,
\]
or equivalently
\[
T_k^{\alpha}u=-\omega^2\mathcal{L}_\alpha^{-1}\rho u\;\; \text{as elements of}\; H^1_{\#}(\alpha, Y)^3.
\]
Now we aim to find the $k$ values for a given $\alpha$ such that $(T_k^{\alpha})^{-1}$ exists and write
\begin{equation}
\label{a-1}
\frac{1}{\omega^2}u=(T_k^{\alpha})^{-1}(-\mathcal{L}_\alpha)^{-1}\rho u.
\end{equation}
This is equivalent to solving the original eigenvalue problem \eqref{I1} on writing
\begin{equation}
\label{representation}
    -\nabla\cdot (\mathbf{C}(x)\mathcal{E}u)=-\mathcal{L}_\alpha T_k^{\alpha } u
\end{equation}

and noting
\begin{equation}
\label{a-2}
\mathcal{L}_\alpha T_k^{\alpha}u=\omega^2  \rho u\;\text{as elements of}\;L^2_{\#}(\alpha,Y)^3.
\end{equation}
Thus we will consider the operator $(T_k^{\alpha})^{-1}(-\mathcal{L}_\alpha)^{-1}\rho$ of $\eqref{a-1}$ and show that for a given subset of $k$ in $\mathbb{C}$ it is a bounded operator from $L^2_{\#}(\alpha, Y)^3$ to $H^1_{\#}(\alpha, Y)^3$. 
In order to accomplish this we express $(T_k^{\alpha})$ explicitly and discern the values $k$ in $\mathbb{C}$ for which $(T_k^{\alpha})$ is invertible. 

We start by decomposing  $H^1_{\#}(\alpha, Y)^3$ into invariant subspaces of source free modes associated with a quasi-periodic resonance spectra. This decomposition provides the explicit spectral representation for the operator $(T_k^{\alpha})$, see Theorem \ref{Ht2.5}. Consider the quasi-periodic case given by $\alpha\in Y^*\setminus\{0\}$. Set {$W^\alpha_1=\{u\in H^1_{\#}(\alpha, Y)^3:\mathcal{E}(u)=0\; \text{in}\; D\}$} and {$W^\alpha_2=\{u\in H^1_{\#}(\alpha, Y)^3:\mathcal{E}(u)=0\; \text{in}\; Y\setminus D\}$}. One checks that these spaces are orthogonal in the $\langle\cdot,\cdot\rangle$ inner product. We define $W^\alpha_3:=(W^\alpha_1\oplus W^\alpha_2)^\perp$ and
\begin{equation}
\label{H8}
H^1_{\#}(\alpha, Y)^3=W^\alpha_1\oplus W^\alpha_2\oplus W^\alpha_3.
\end{equation}

Now consider  $\alpha=0$ and decompose $H^1_{\#}(0, Y)^3$. Set {$W^0_1=\{u\in H^1_{\#}(0, Y)^3:\mathcal{E}(u)=0\; \text{in}\; D\}$} and $W^0_2=\{u\in H^1_{\#}(0, Y)^3:\mathcal{E}(u)=0\; \text{in}\; Y\setminus D\}$. One checks that these spaces are orthogonal in the $\langle\cdot,\cdot\rangle$ inner product. Set $\langle\rho\rangle=\int_Y\,\rho\,dx$, then one also has the equivalent representation of $W^0_2$ given by
\begin{lemma}
\label{W2}
The subspace $W^0_2$  of  $H^1_{\#}(0, Y)^3$ has the representation
\begin{equation}
    \label{H9}
    W^0_2=\{u=\tilde{u}-\langle \rho\rangle^{-1}\int_{D}\rho^1\,\tilde{u}\,dx\;1_Y\;|\;\tilde{u}\in \tilde{H}^1_0(D)^3\}
\end{equation}
where $\tilde{H}^1_0(D)^3$ is the subspace of $H_0^1(Y)^3$ given by all $H^1_0(D)^3$  functions extended by zero into $Y\setminus D$ and and $1_Y$ is the indicator function of $Y$.
\end{lemma}
\noindent This Lemma is proved in the Appendix.
Clearly $W^0_1$ and $W^0_2$ are orthogonal subspaces of $ H^1_{\#}(0, Y)^3$ and define $W^0_3:=(W^0_1\oplus W^0_2)^\perp$ and
\begin{equation}
\label{H10}
H^1_{\#}(0, Y)^3=W^0_1\oplus W^0_2\oplus W^0_3.
\end{equation}

With these definitions in hand we write $W^\alpha_1$, $W^\alpha_2$, $W^\alpha_3$ for all $\alpha\in Y^\ast$.
To set up the spectral analysis we  observe that orthogonality and integration by parts shows that 
for $u\in W^\alpha_3$, 
\begin{equation}\label{harmonic}
    \mathcal{L}u=0
\end{equation}
separately in $D$ and $Y\setminus D$, for all $\alpha\in Y^\ast$ and this
implies that elements of $W^\alpha_3$ can be represented in terms of single layer potentials supported on $\partial D$. We introduce the 3-dimensional $\alpha$-quasi-periodic Green's function
\begin{equation}
\label{H11}
\mathbf{G}^{\alpha}(x,y)=\frac{1}{\mu_1}\sum_{n\in\mathbb{Z}^3}e^{i(2\pi n+\alpha)\cdot(x-y)}\Big(\frac{-\delta_{ij}}{|2\pi n+ \alpha|^2}+\frac{\lambda_1+\mu_1}{\lambda_1+2\mu_1}\frac{(2\pi n+ \alpha)_i(2\pi n+ \alpha)_j}{|2\pi n+ \alpha|^4}\Big),
\end{equation}
and for $\alpha=0$ the periodic Green's function by 
\begin{equation}
\label{H12}
\mathbf{G}^0(x,y)=\frac{1}{\mu_1}\sum_{n\in\mathbb{Z}^3\setminus \{0\}}e^{i(2\pi n)\cdot(x-y)}\Big(\frac{-\delta_{ij}}{|2\pi n|^2}+\frac{\lambda_1+\mu_1}{\lambda_1+2\mu_1}\frac{4\pi^2 n_in_j}{|2\pi n|^4}\Big)\;\text{for}\;\alpha=0.
\end{equation}
Let $H^{1/2}(\partial D)^3$ be the fractional Sobolev space on $\partial D$ with dual  $(H^{1/2}(\partial D)^3)^*=H^{-1/2}(\partial D)^3$. For $\phi\in H^{-1/2}(\partial D)^3$, and $\alpha\in Y^*$ define the single layer potential $\mathcal{S}^\alpha_{D}[\phi](x)$  associated with the Lam\'e system
\begin{equation}
\label{H13}
\mathcal{S}^\alpha_{D}[\phi](x):=\int_{\partial D}\mathbf G^{\alpha}(x,y)\phi(y)\;ds(y)\:\:,
x\in Y.
\end{equation}
It follows from \cite{AmmariKangLee}, for any $\phi\in H^{-1/2}(\partial D)^3$
\begin{equation}
\label{H14}
\begin{aligned}
\mathcal{L}\mathcal{S}^\alpha_D\phi&=0\; \text{in}\;D\; \text{and}\; Y\setminus D,\\
\mathcal{S}^\alpha_D\phi|_{\partial D}^- &=\mathcal{S}^\alpha_D\phi|_{\partial D}^+,\\
\frac{\partial}{\partial_\nu}\mathcal{S}^\alpha_D\phi|_{\partial D}^{\pm}&=(\pm\frac{1}{2}I+\mathcal{(\tilde{K}^{-\alpha})^*})[\phi],
\end{aligned}
\end{equation}
where $\nu$ is the outward unit normal to $\partial D$ and $(\tilde{\mathcal{K}}^{-\alpha}_D)^*$ is the Neumann Poincar\'e operator defined by
\begin{equation}
\label{H15}
(\tilde{\mathcal{K}}^{-\alpha}_D)^*[\phi](x)=p.v\int_{\partial D}\partial_{\nu_{x}}\mathbf{G}^\alpha(x-y)\phi(y)\;ds(y),\;\; x\in\partial D
\end{equation}
where $\tilde{\mathcal{K}}^{\alpha}_D$ is the Neumann Poincar\'e operator
\begin{equation}
\label{H16}
\tilde{\mathcal{K}}^{\alpha}_D[\phi](x)=p.v\int_{\partial D}\partial_{\nu_{y}}\mathbf{G}^\alpha(x-y)\phi(y)\;ds(y),\;\; x\in\partial D.
\end{equation}
Define $\mathcal{S}^\alpha_{\partial D}\phi=\mathcal{S}^\alpha_D\phi|_{\partial D}$ for all $\phi\in H^{-1/2}(\partial D)^3$, then we have the following Lemma.
\begin{lemma}
$\mathcal{S}^\alpha_{\partial D}:H^{-1/2}(\partial D)^3\mapsto H^{1/2}(\partial D)^3$ is invertible. 
\end{lemma}
\begin{proof}
We show that $Ker\{\mathcal{S}^\alpha_{\partial D}\}=\{0\}$. Let $\phi\in H^{-1/2}(\partial D)^3$ and suppose $\mathcal{S}^\alpha_{\partial D}\phi=0$. Set $u=\mathcal{S}^\alpha_D\phi$.  Then $u\in W^\alpha_3$, and satisfy $\mathcal{L}u=0,\; \text{in}\; D$ with $u|_{\partial D^-}=0$. Hence we have $u=0$ in $D$. Likewise since $\mathcal{L}u=0,\; \text{in}\; Y\setminus D$ and $u|_{\partial D^+}=0$ with quasi-periodic boundary conditions on $\partial Y$, we conclude $u=0$ in $Y\setminus D$. Then since $\phi=\partial_nu|_{\partial D^+}-\partial_nu|_{\partial D^-}$, it yields $\phi=0$. To show the surjectivity of $\mathcal{S}^\alpha_{\partial D}$, let $g\in H^{1/2}(\partial D)^3$. Since the trace map $G:W^\alpha_3\mapsto H^{1/2}(\partial D)^3$ is onto, there exists $u\in W^\alpha_3$ such that $Gu=g$. Define $\phi_u:=\partial_nu|_{\partial D^+}-\partial_nu|_{\partial D^-}$ and $w(x):=\mathcal{S}^\alpha_{D}[\phi_u](x)=\int_{\partial D}\mathbf G^{\alpha}(x,y)\phi_u(y)\;ds(y)$. Thus $u,w\in W^\alpha_3$ and $l:=u-w$ satisfy
\begin{align*}
    \mathcal{L}l &=0\; \text{in}\; D\;\text{and}\; Y\setminus D,\\
    l|_{\partial D^-}&=  l|_{\partial D^+},\\
    n\cdot\mathbf{C}^1\mathcal{E} (l)|_{\partial D^-}&=n\cdot\mathbf{C}^1\mathcal{ E}(l)|_{\partial D^+}.
\end{align*}
Then we see that
\[
\int_Y \mathbf{C}^1\mathcal{E}(l):\mathcal{E}(l)\;dx=0.
\]
This implies that $l$ is a rigid body motion and therefore we conclude that $l=0$. Hence $u=w$.
\end{proof}
Let $\mathbf{\Gamma}=(\Gamma_{j,k})_{j,k=1}^3$ denote the Kelvin matrix associated with the fundamental solution of the Lam\'e operator and has the following representation 
\begin{equation}
\label{H17}
\Gamma_{j,k}(x)=-\frac{b_1\delta_{jk}}{4\pi|x|}-\frac{b_2}{4\pi}\frac{x_jx_k}{|x|^3}
\end{equation}
with 
\begin{equation}
\label{H18}
b_1:=\frac{1}{2}\Big(\frac{1}{\mu_1}+\frac{1}{2\mu_1+\lambda_1}\Big)\; \text{and}\;\; b_2:=\frac{1}{2}\Big(\frac{1}{\mu_1}-\frac{1}{2\mu_1+\lambda_1}\Big).
\end{equation}
We define $\mathbf{F}(x,y):=\mathbf{G}^\alpha(x,y)-\mathbf{\Gamma}(x,y)$. One has the identities
\begin{equation}
\label{H19}
\mathcal{L}\mathbf{G}^\alpha(x,y) =\sum_{n\in\mathbb{Z}^3}\delta(x-y-n)I\; , \alpha\neq 0,
\end{equation}
\begin{equation}
\label{H20}
\mathcal{L}\mathbf{G}^0(x,y) =\sum_{n\in\mathbb{Z}^3}\delta(x-y-n)I-I.
\end{equation}

We also know that
\begin{equation}
\label{a1}
-\mathcal{L}\mathbf{\Gamma}(x,y)=\delta(x-y)I.
\end{equation}
Thus on $Y$ for $\alpha\neq 0$, we have 
\begin{equation}
\label{a2}
\begin{cases}
-\mathcal{L}\mathbf{F}(x,y)=0\;\text{for}\;x\;\text{in}\;Y\\

\mathbf{F}(x,y)|_{\partial Y}\text{is continuous.}
\end{cases}
\end{equation}
\\
This shows that $\mathbf{F}$ satisfies the homogeneous Lam\'e equation on $Y$ and we form
\begin{equation}
\label{a3}
\partial_{\nu_x}\mathbf{G}^\alpha(x,y)=\partial_{\nu_x}\mathbf{\Gamma}(x,y)+\partial_{\nu_x}\mathbf{F}(x,y).
\end{equation}
So the Neumann Poincar\'e operator has the equivalent representation
\begin{align}\label{hilbertkernel}
(\tilde{\mathcal{K}}^{-\alpha}_D)^*[\phi](x)=p.v\int_{\partial D}\partial_{\nu_{x}}\mathbf{G}^\alpha(x,y)\phi(y)\;ds(y)&=p.v\{\int_{\partial D}\partial_{\nu_{x}}\mathbf{\Gamma}{(x,y)}\phi(y)\;ds(y)\\
&+\int_{\partial D}\partial_{\nu_{x}}\mathbf{F}(x,y)\phi(y)\;ds(y)\}.\nonumber
\end{align}

Recall the identity
\begin{equation}
\label{H21}
\partial_{\nu_{x}}\mathbf{\Gamma(x,y)}=k_0\mathbf{K}_1(x,y)+\mathbf{K}_2(x,y)
\end{equation}

where 
\begin{equation}
\label{H22}
k_0=\frac{-\mu_1}{2(2\mu_1+\lambda_1)},
\end{equation}

\begin{align}
\label{H23}
\mathbf{K_1}(x,y)&=\frac{n_x(x-y)^T-(x-y)n_x^T}{2\pi|x-y|^3},\\
\label{H24}
\mathbf{K_2}(x,y)&=\frac{\mu_1}{2\mu_1+\lambda_1}\frac{(x-y)\cdot n_y}{4\pi|x-y|^3}I+\frac{2(\mu_1+\lambda_1)}{2\mu_1+\lambda_1}\frac{(x-y)\cdot n_y}{4\pi|x-y|^5}(x-y)(x-y)^T.
\end{align}
Here $I$ is the $3\times 3$ identity matrix.
Then by substituting this to \eqref{a3} yields
\begin{equation}
\label{a4}
\partial_{\nu_x}\mathbf{G}^\alpha(x,y)=k_0\mathbf{K}_1(x,y)+\mathbf{K}_2(x,y)+\partial_{\nu_x}\mathbf{F}.
\end{equation}
We now argue as in the fundamental paper  \cite{AndoKangMiyanishi}. Define 
\begin{equation}
\label{H25}
\mathbf{T}[\phi](x)=p.v\int_{\partial D}\mathbf{K}_1(x,y)\phi(y)\;ds(y),\;\;x\in \partial D.
\end{equation}
Since $\mathbf{K}_2$ and $\mathbf{F}$ satisfy the weakly singular conditions $|\mathbf{K}_2(x,y)|\leq C|x-y|^{-1}$ and $|\partial_{n(x)}\mathbf{F}(x,y)|\leq C|x-y|^{-1}$ the integral operators  $\int_{\partial D}\mathbf{K}_2(x,y)\phi(y)\;ds(y)$ and $\int_{\partial D}\partial_{n(x)}\mathbf{F}_2(x,y)\phi(y)\;ds(y)$ are compact on $H^{-1/2}(\partial D)^3$. Therefore we write
\[
(\tilde{\mathcal{K}}^{-\alpha}_D)^*[\phi](x)=k_0\mathbf{T}[\phi]+\text{compact operator}.
\]
so $(\tilde{\mathcal{K}}^{-\alpha}_D)^*-k_0\mathbf{T}$ is compact on $Y$ as elements of $H^{-1/2}(\partial D)^3$ to conclude that 
\[
p_3((\tilde{\mathcal{K}}^{-\alpha}_D)^*):=((\tilde{\mathcal{K}}^{-\alpha}_D)^*)^3-k_0^2(\tilde{\mathcal{K}}^{-\alpha}_D)^*,
\]
is compact.

The spectrum of $(\tilde{\mathcal{K}}^{-\alpha}_D)^*$ is denoted by $\sigma((\tilde{\mathcal{K}}^{-\alpha}_D)^*)$. We conclude  using the spectral mapping theorem  that $p_3(\sigma((\tilde{\mathcal{K}}^{-\alpha}_D)^*)))=\sigma(p_3((\tilde{\mathcal{K}}^{-\alpha}_D)^*))$. Since $p_3((\tilde{\mathcal{K}}^{-\alpha}_D)^*)$ is compact, $p_3(\sigma((\tilde{\mathcal{K}}^{-\alpha}_D)^*))$ consists of eigenvalues (of finite multiplicities) converging to $0$.
Let $\{r_n(\alpha)\}$ be the set of eigenvalues of $p_3((\tilde{\mathcal{K}}^{-\alpha}_D)^*)$ and suppose $\zeta\in\sigma((\tilde{\mathcal{K}}^{-\alpha}_D)^*)$. Then for each $n$, we have the equation\[ \zeta^3-k_0^2\zeta=r_n(\alpha).\]
Now set $\gamma_n^\pm=\Big[\frac{1}{2}(-27r_n(\alpha)\pm \sqrt{729r^2_n+108k^6_0})\Big]^{1/3}$. By solving this cubic polynomial we get three roots and denote them by $\zeta_{n_1}, \zeta_{n_2}, \zeta_{n_3}$ where
\begin{equation}
  \label{roots}  
\begin{aligned}
\zeta_{n_1}&=-\frac{1}{3}\gamma_n^{+}-\frac{1}{3}\gamma_n^{-}\\
\zeta_{n_2}&=\frac{1-i\sqrt 3}{6}\gamma_n^{+}+\frac{1+i\sqrt 3}{6}\gamma_n^{-}\\
\zeta_{n_3}&=\frac{1+i\sqrt 3}{6}\gamma_n^{+}+\frac{1-i\sqrt 3}{6}\gamma_n^{-}.
\end{aligned}
\end{equation}

Thus for each $n$ we get $3$ different eigenvalues and therefore we will have $3$ different sequences of eigenvalues with accumulation points $0, k_0$ and $-k_0$.
We emphasize that while not indicated explicitly these eigenvalues depend on $\alpha\in Y^\ast$.  We summarize the results below.

\begin{lemma}
\label{neumannPoincareSpectra}
The point spectra of $(\tilde{\mathcal{K}}^{-\alpha}_D)^*$ is given by the three sequences of eigenvalues given by \eqref{roots} and the sequences converge to the three accumulation points $0$, $k_0$, and $-k_0$.
\end{lemma}

Let $G:\,W^\alpha_3\rightarrow H^{-1/2}(\partial D)^3$ be the trace operator which is bounded one to one and onto.
\begin{lemma}
\label{Hl2.2}
 $\mathcal{S}^\alpha_D:H^{-1/2}(\partial D)^3\mapsto W^\alpha_3$ is one to one and onto and moreover $(\mathcal{S}^\alpha_D)^{-1}=(\mathcal{S}^\alpha)^{-1}_{\partial D} G$.
\end{lemma}

\begin{proof}

Let $\phi\in H^{-1/2}(\partial D)^3$.  As shown earlier $\mathcal{S}^\alpha_D:H^{-1/2}(\partial D)^3\mapsto W^\alpha_3$ has $Ker(\mathcal{S}^\alpha_D)=0$ so $\mathcal{S}^\alpha_D$ is one to one and $f=\mathcal{S}_D^\alpha\phi\in W^\alpha_3$ for every $\phi\in H^{-1/2}(\partial D)^3$. Now suppose $u\in W^\alpha_3$, and consider $Gu=u|_{\partial D}\in   H^{1/2}(\partial D)^3$. Define $w=\mathcal{S}^\alpha_D(\mathcal{S}^\alpha)^{-1}_{\partial D} Gu)$. Since $u, w\in W^\alpha_3$, we have $u-w\in W^\alpha_3$. Also $Gu=Gw$ and hence $G(u-w)=0$, so $w-u\in (W^\alpha_1\oplus W^\alpha_2)$. But $W^\alpha_3=(W^\alpha_1\oplus W^\alpha_2)^{\perp}$ and therefore $w=u$.
\end{proof}

We define an auxiliary operator $T: W^\alpha_3\mapsto W^\alpha_3$ such that
\begin{equation}
\label{H27}
\langle Tu,v\rangle =\frac{1}{2}\int_{Y\setminus D}\mathbf{C}^1\mathcal{E}(u):\overline{\mathcal{E}(v)}\; dx -\frac{1}{2}\int_{D}\mathbf{C}^1\mathcal{E}(u):\overline{\mathcal{E}(v)}\;dx.
\end{equation}
If $\xi$ is an eigenvalue of $T$, then there exists $u\in W^\alpha_3$ such that 
\begin{equation}
    \label{specq}
    Tu=\xi u.
\end{equation}
From $\eqref{H27}$, it is clear that,
\begin{align*}
-\frac{1}{2}\langle u,u\rangle &=-\frac{1}{2}\int_{{Y\setminus D}}\mathbf{C}^1\mathcal{E}(u):\overline{\mathcal{E}(u)}\;dx-\frac{1}{2}\int_{D}\mathbf{C}^1\mathcal{E}(u):\overline{\mathcal{E}(u)}\; dx \\
&\leq\langle Tu,u\rangle=\xi \langle u,u\rangle\\
&\leq \frac{1}{2}\int_{{Y\setminus D}}\mathbf{C}^1\mathcal{E}(u):\overline{\mathcal{E}(u)}\;dx+\frac{1}{2}\int_{D}\mathbf{C}^1\mathcal{E}(u):\overline{\mathcal{E}(u)}\;dx=\frac{1}{2}\langle u,u\rangle.
\end{align*}
Thus for any eigenvalue $\xi$ of $T$, we have the following.
\[
-\frac{1}{2}\leq\xi\leq\frac{1}{2}.
\]
The upper bound $1/2$ is the eigenvalue associated with the eigenspace $W_1^\alpha$.

The next theorem shows the relation between the auxiliary operator $T$ restricted to $W^\alpha_3$ and the elastic Neumann Poincar\'e  operator.

\begin{theorem}
\label{Ht2.3}
For $u\in W^\alpha_3$ the operator $T$ is given by 
\begin{equation}
\label{H28}
T=\mathcal{S}^\alpha_{D}(\tilde{\mathcal{K}}^{-\alpha}_{D})^*(\mathcal{S}_{D}^\alpha)^{-1}
\end{equation} and is self-adjoint.
\end{theorem}
\begin{proof}
Let $u,v\in W^\alpha_3$. Consider
\begin{equation}
\label{H29}
    \langle \mathcal{S}^\alpha_{D}(\tilde{\mathcal{K}}^{-\alpha}_{D})^*(\mathcal{S}_{D}^\alpha)^{-1}u,v\rangle= \int_{Y}\mathbf{C}^1\mathcal{E}[\mathcal{S}^\alpha_{D}(\tilde{\mathcal{K}}^{-\alpha}_{D})^*(\mathcal{S}^\alpha_{D})^{-1}u]:\overline{\mathcal{E}(v)}\;dx
\end{equation}
\begin{equation*}
    \begin{aligned}
    &=\int_{Y\setminus D}\mathbf{C}^1\mathcal{E}[\mathcal{S}^\alpha_{D}(\tilde{\mathcal{K}}^{-\alpha}_{D})^*(\mathcal{S}_{D}^\alpha)^{-1}u]:\overline{\mathcal{E}(v)}\;dx\\
    &+\int_{D}\mathbf{C}^1\mathcal{E}[\mathcal{S}^\alpha_{D}(\tilde{\mathcal{K}}^{-\alpha}_{D})^*(\mathcal{S}_{D}^\alpha)^{-1}u]:\overline{\mathcal{E}(v)}\;dx.
    \end{aligned}
\end{equation*}
Since $\mathcal{L}\mathcal{S}^\alpha_D\phi=0$ in $D$ and $Y\setminus D$, for any $\phi\in H^{-1/2}(\partial D)^3$, using integration by parts gives
\[
 \langle \mathcal{S}^\alpha_{D}(\tilde{\mathcal{K}}^{-\alpha}_{D})^*(\mathcal{S}_{D})^\alpha)^{-1}u,v\rangle=\int_{\partial D}(n\cdot \mathbf{C}^1\mathcal{E}[\mathcal{S}^\alpha_{D}(\tilde{\mathcal{K}}^{-\alpha}_{D})^*(\mathcal{S}_{D}^\alpha)^{-1}u]|_{\partial D^-}-n\cdot \mathbf{C}^1\mathcal{E}[\mathcal{S}^\alpha_{D}(\tilde{\mathcal{K}}^{-\alpha}_{D})^*(\mathcal{S}_{D}^\alpha)^{-1}u]|_{\partial D^+})\overline{v}.
\]
Applying the jump conditions from $\eqref{H14}$ gives
\begin{equation}
\label{H30}
   \langle \mathcal{S}^\alpha_{D}(\tilde{\mathcal{K}}^{-\alpha}_{D})^*(\mathcal{S}_{D}^\alpha)^{-1}u,v\rangle=-\int_{\partial D} (\tilde{\mathcal{K}}^{-\alpha}_{D})^*(\mathcal{S}_{D}^\alpha)^{-1}u\overline{v}\;ds.
\end{equation}
The same jump conditions gives
\[
(\tilde{\mathcal{K}}^{-\alpha}_{D})^*[\phi]=\frac{1}{2}(\frac{\partial}{\partial_\nu}\mathcal{S^\alpha}_D\phi|_{\partial D^-}+\frac{\partial}{\partial_\nu}\mathcal{S}^\alpha_D\phi|_{\partial D^+})
\]
where $\phi=\mathcal({S}^\alpha_{D})^{-1}u$ . Thus $\eqref{H30}$ yields
\begin{equation}
\label{H31}
     \langle \mathcal{S}^\alpha_{D}(\tilde{\mathcal{K}}^{-\alpha}_{D})^*(\mathcal{S}_{D}^\alpha)^{-1}u,v\rangle=-\int_{\partial D}(\frac{1}{2}\frac{\partial}{\partial_\nu}u|_{\partial D^-}+\frac{1}{2}\frac{\partial}{\partial_\nu}u|_{\partial D^+})\overline{v}\;ds
\end{equation}
\begin{equation*}
    \begin{aligned}
    &=-\frac{1}{2}\int_{D}\mathbf{C}^1\mathcal{E}(u):\overline{\mathcal{E}(v)}\;dx-\frac{1}{2}(-\int_{Y\setminus D}\mathbf{C}^1\mathcal{E}(u):\overline{\mathcal{E}(v)}\;dx)\\
    &=\langle Tu,v\rangle.
    \end{aligned}
\end{equation*}
\end{proof}

\begin{lemma}
\label{Hl2.4}
When restricted to $W^\alpha_3$ the point spectrum of $T$ is given by the point spectrum of $(\tilde{\mathcal{K}}^{-\alpha}_{D})^*$ and the essential spectrum of $T$ is given by the accumulation points $0$, $k_0$, and $-k_0$ of the eigenvalues of operator $(\tilde{\mathcal{K}}^{-\alpha}_{D})^*$.
\end{lemma}
\begin{proof}
Since $T$ is self-adjoint, it has empty residual spectrum and $\sigma(T)=\sigma_p(T)\cup\sigma_{ess}(T)$, see, e.g., \cite{ReedSimonV1}. We prove the theorem by showing $\sigma_p(T)=\sigma_p((\tilde{\mathcal{K}}^{-\alpha}_{D})^*)$ and $\sigma_{ess}(T)=\{-k_0,0,k_0\}$. To prove the claim on the point spectrum, suppose $(\xi,u)\in(-1/2,1/2]\times W^\alpha_3$ satisfies $Tu=\xi u.$ Then \begin{align*}
\mathcal{S^\alpha}_{D}(\tilde{\mathcal{K}}^{-\alpha}_{D})^*(\mathcal{S}_{D}^\alpha)^{-1}u&=\xi u\\
(\tilde{\mathcal{K}}^{-\alpha}_{D})^*(\mathcal{S}_{D}^\alpha)^{-1}u&=\xi (\mathcal{S}_{D}^\alpha)^{-1}u.
\end{align*}
This shows that $(\mathcal{S}^\alpha_{D})^{-1}u$ is an eigenfunction for $(\tilde{\mathcal{K}}^{-\alpha}_{D})^*$ associated the with eigenvalue $\xi.$ On the other hand, if $(\xi,w)\in(-1/2,1/2]\times H^{-1/2}(\partial D)^3$ satisfy $(\tilde{\mathcal{K}}^{-\alpha}_{D})^*w=\xi w$, then since the trace map is onto, there is a $u\in W^\alpha_3$ such that $w=(\mathcal{S}^\alpha_{D})^{-1}u$. Thus replacing $(\mathcal{S}^\alpha_{D})^{-1}u$ for $w$ yields $(\tilde{\mathcal{K}}^{-\alpha}_{D})^*(\mathcal{S}^\alpha_{D})^{-1}u=\xi (\mathcal{S}^\alpha_{D})^{-1}u$ and therefore we obtain $\mathcal{S}^\alpha_{D}(\tilde{\mathcal{K}}^{-\alpha}_{D})^*(\mathcal{S}_{D}^\alpha)^{-1}u=\xi u$. Which shows that $u$ is an eigenfunction of $T$ associated with the eigenvalue $\xi.$

We now prove $\sigma_{ess}(T)=\{-k_0,0,k_0\}$. Suppose $\xi\in \{-k_0,0,k_0\}$. Note first there exists a sequence $\{\xi_n,\rho_n\}\in \mathbb{R}\times H^{-1/2}(\partial D)^3$ such that $\rho_n$ is a sequence of eigenvectors, $\Vert\rho_n\Vert=1$, and $\xi_n\rightarrow\zeta$ for which \[|\xi_n-\xi|=\|((\tilde{\mathcal{K}}^{-\alpha}_{D})^*-\xi I)\rho_n\|.\]
Now since $(\mathcal{S}_D^\alpha)^{-1}$ is onto, we can find $u_n\in W^\alpha_3$ such that $\rho_n=(\mathcal{S}_D^\alpha)^{-1}u_n$. Therefore
\begin{align*}
  \|(T-\xi I)u_n\|&=\|  \mathcal{S}^\alpha_{D}((\tilde{\mathcal{K}}^{-\alpha}_{D})^*-\xi I)(\mathcal{S}_{D}^\alpha)^{-1}u_n\|\\
  &=\|  \mathcal{S^\alpha}_{D}((\tilde{\mathcal{K}}^{-\alpha}_{D})^*-\xi I)\rho_n\|\\
  &\leq\| \mathcal{S}^\alpha_{D}\|\|((\tilde{\mathcal{K}}^{-\alpha}_{D})^*-\xi I)\rho_n\|\\
  &\leq M\|\xi_n-\xi|.
\end{align*}
Since $T$ is selfadjoint the eigenfunctions $\{u_n\}$ form an orthogonal system and $\{-k_0,0,k_0\}$ constitute the essential spectrum of $T$.
\end{proof}


From Lemma \eqref{Hl2.4}, we see that the eigenvalues of the elastic NP operator lie in $(-1/2,1/2]$. The accumulation points of eigenvalues $k_0,-k_0,0$ also lie in $(-1/2,1/2]$. To see that $k_0,-k_0$ lie in $(-1/2,1/2]$, we use the relation $K_1-\frac{2\mu_1}{3}=\lambda_1$ where $K_1>0$ is the bulk modulus. Since
\[
\pm k_0=\mp\frac{\mu_1}{2(2\mu_1+\lambda_1)},
\]
and we have
\begin{align*}
    |k_0|=\frac{\mu_1}{2(2\mu_1+\lambda_1)}&=\frac{\mu_1}{4\mu_1+2\lambda_1}\\
    &=\frac{\mu_1}{4\mu_1+2(K_1-2\mu_1/3)}=\frac{\mu_1}{8\mu_1/3+2K}\\
    &\leq\frac{\mu_1}{8\mu_1/3}=\frac{3}{8}<\frac{1}{2}.
\end{align*}

We now derive an appropriate resolution of the identity on the space $W^\alpha_3$ and an associated spectral representation formula for $T$.  In what follows we do not explicitly indicate dependence of eigenvalues on $\alpha\in Y^\ast$ for ease of exposition.  Let the three sequences of eigenvalues for $T$ be denoted by $\{\xi_i^k\}_{i=1}^\infty$, $k=1,2,3$, associated with the three accumulation points $\xi^{1}=-k_0$, $\xi^2=0$, and $\xi^3=k_0$. The invariant subspace associated with each eigenvalue is denoted by $E_i^k=\{u\in W^\alpha_3\,;Tu=\xi_i^k u\}$
and the orthogonal projection onto this subspace is denoted by $P_i^k$, here orthogonality is with respect to the $\langle\cdot,\cdot\rangle$ inner product. We write $E^1=\sum_{i=1}^\infty \oplus E_i^1$, $E^2=\sum_{i=1}^\infty \oplus E_i^2$, $E^3=\sum_{i=1}^\infty \oplus E_i^3$, $E^4=ker\{\tilde{T}\}$, and the projection operators onto the spaces are  $P^k$, $k=1,2,3$.
We define
\begin{align}
    \label{component1}
    \tilde{T}=\sum_{k=1}^3\sum_{i=1}^\infty (\xi_i^k-\xi^k)P^k_i,
\end{align}
and $E^4=ker\{\tilde{T}\}$, with projection $P^4$.
This operator is selfadjoint and compact. Compactness follows as it is the limit of the rank one operators
\begin{align}
    \label{component2}
    \tilde{T}^n=\sum_{k=1}^3\sum_{i=1}^n (\xi_i^k-\xi^k)P^k_i.
\end{align}
One has the following decomposition of $T$ restricted to $W^\alpha_3$:
\begin{lemma}
\label{decompofT}
\begin{align}
    \label{component3}
    {T}=\sum_{k=1}^3 \sum_{i=1}^\infty \xi_i^k P^k_i.
\end{align}
The identity on $W_3$ is given by
\begin{align}
    \label{resolution}
    I=\sum_{k=1}^3\sum_{i=1}^\infty P_i^k+P^4.
\end{align}
\end{lemma}
\begin{proof}

Since $\tilde{T}$ is compact and selfadjoint on $W^\alpha_3$, it follows immediately from the theory that $W^\alpha_3=(\sum_{k=1}^3\sum_{i=1}^\infty\oplus E_i^k)\oplus ker\{\tilde{T}\}$ and \eqref{resolution} follows. 
Writing any element of $u\in W^\alpha_3$ as $u=\sum_{k=1}^3\sum_{i=1}^\infty P_i^k u +P^4 u$ shows
\begin{equation}
    \label{sn to T}
    \Vert T-(-k_0 P^1+k_0P^3)-\tilde{T}^n\Vert\rightarrow 0,
\end{equation}
in the operator norm.
Thus given $\epsilon>0$ we can find $N$ such that
\begin{equation}
    \label{triangle}
    \Vert T-(-k_0 P^1+k_0P^3)-\tilde{T}\Vert\leq\Vert T-(-k_0 P^1+k_0P^3)-\tilde{T}^n\Vert+\Vert \tilde{T}^n-\tilde{T}\Vert\leq \epsilon,
\end{equation}
for all $n>N$ 
hence
\begin{equation}
\label{decpT}
T-(-k_0 P^1+k_0P^3)=\tilde{T}.
\end{equation}
and the Lemma  follows.
\end{proof}

To simplify the exposition we collect all eigenvalues of $T$ restricted to $W^\alpha_3$ and denote them as the sequence  $\{\tau_n(\alpha)\}_{n=1}^\infty\in (-1/2,1/2)$ and we have the definition

\begin{definition}\label{structural}
The structural spectra of the crystal is defined as $\cup_{\alpha\in Y^\ast}\{\tau_i(\alpha)\}_{i=1}^{\infty}$.
\end{definition}
\noindent This spectra is independent of contrast encodes the geometry of the crystal.

For fixed $\alpha$ the projections onto their eigenspaces in $W_3^\alpha$ are denoted as $P_{\tau_n(\alpha)}$. We also write $\tau_0=0$ and the projection onto $ker\{T\}$ as $P_{\tau_0}$.
On writing the projections on $W^\alpha_1$ and $W^\alpha_2$ as $P^\alpha_1$, $P^\alpha_2$ respectively, we arrive at the desired partition of unity for  $u$, $v$ in $H^1_\#(\alpha,Y)^3=W^\alpha_1\oplus W^\alpha_2\oplus W^\alpha_3$ given by
\begin{equation}
\label{partition of 1}
\langle u,v\rangle=\langle P^\alpha_1u + P^\alpha_2u+ (\sum_{-\frac{1}{2}<\tau_i(\alpha)<\frac{1}{2}}P_{\tau_i(\alpha)})u,v\rangle.
\end{equation}
The spectral decomposition for $T_{k}^\alpha$ associated with the sesquilinear form is given by

\begin{theorem}
\label{Ht2.5}
The linear operator $T_k^{\alpha}:H^1_{\#}(\alpha, Y)^3\mapsto H^1_{\#}(\alpha, Y)^3$ associated with the

sesquilinear form $B_k$ is given by
\[ \langle T_{k}^\alpha u,v\rangle=\langle k P^\alpha_1u+P^\alpha_2u+\sum_{-\frac{1}{2}<\tau_i(\alpha)<\frac{1}{2}}[k(1/2+\tau_i(\alpha))+(1/2-\tau_i(\alpha))]P_{\tau_i(\alpha)}u,v\rangle
\]

for all $u,v\in H^1_{\#}(\alpha, Y)^3$.
\end{theorem}

\begin{proof}
Recall that for all  $u,v\in H^1_\#(\alpha,Y)^3$
\begin{align*}
B_{k}(u,v)&=\int_{Y}{\mathbf{C}(x)\mathcal{E}( u):\overline{\mathcal{E}( v)}\; dx}\\
&=\int_{Y\setminus D}\mathbf{C}^2\mathcal{E}( u):\overline{\mathcal{E}(v)}\; dx+\int_{ D}\mathbf{C}^1\mathcal{E}(u):\overline{\mathcal{E}(v)}\; dx\\
&=k\int_{Y\setminus D}\mathbf{C}^1\mathcal{E}( u):\overline{\mathcal{E}(v)}\; dx+\int_{ D}\mathbf{C}^1\mathcal{E}(u):\overline{\mathcal{E}(v)}\; dx.
\end{align*}
Let $u,v\in W_3$. Since $P_{\tau_i(\alpha)}u$ is the eigenvector corresponding to eigenvalue $\tau_i(\alpha)$ of $T$, we have $\langle T(P_{\tau_i(\alpha)}u),v\rangle=\tau_i(\alpha)\langle P_{\tau_i(\alpha)}u,v\rangle$ or equivalently after manipulation
\begin{align*}
  \int_{Y\setminus D}\mathbf{C}^1\mathcal{E} (P_{\tau_i(\alpha)}u):\overline{\mathcal{E}(v)}\; dx
  &=\frac{(\frac{1}{2}+\tau_i(\alpha))}{(\frac{1}{2}-\tau_i(\alpha))}\int_{D}\mathbf{C}^1\mathcal{E} (P_{\tau_i(\alpha)}u):\overline{\mathcal{E}(v)}\; dx.
\end{align*}
So we find that
\begin{align*}
    B_{k}(P_{\tau_i(\alpha)}u,v)&=[k\frac{(\frac{1}{2}+\tau_i(\alpha))}{(\frac{1}{2}-\tau_i(\alpha))}+1]\int_{D}\mathbf{C}^1\mathcal{E} (P_{\tau_i(\alpha)}u):\overline{\mathcal{E}(v)}\; dx.
\end{align*}
We also have 
\[
\int_{D}\mathbf{C}^1\mathcal{E} (P_{\tau_i(\alpha)}u):\overline{\mathcal{E}(v)}\; dx=(1/2-\tau_i(\alpha))\int_{Y}\mathbf{C}^1\mathcal{E} (P_{\tau_i(\alpha)}u):\overline{\mathcal{E}(v)}\; dx.
\]
Therefore
\[
 B_{k}(P_{\tau_i(\alpha)}u,v)=[k(1/2+\tau_i(\alpha))+(1/2-\tau_i(\alpha))]\int_{Y}\mathbf{C}^1\mathcal{E} (P_{\tau_i(\alpha)}u):\overline{\mathcal{E}(v)}\; dx.
\]
It is easily seen that
\begin{align*}
 B_{k}(P^\alpha_1u,v) &=k \int_{Y\setminus D}\mathbf{C}^1\mathcal{E} (P^\alpha_1u):\overline{\mathcal{E}(v)}\; dx\\
 B_{k}(P^\alpha_2u,v) &=k \int_{D}\mathbf{C}^1\mathcal{E} (P^\alpha_2u):\overline{\mathcal{E}(v)}\; dx,
 \end{align*}
so
\[ \langle T_{k}^\alpha u,v\rangle=\langle k P^\alpha_1u+P^\alpha_2u+\sum_{-\frac{1}{2}<\tau_i<\frac{1}{2}}[k(1/2+\tau_i(\alpha))+(1/2-\tau_i(\alpha))]P_{\tau_i(\alpha)}u,v\rangle.
\]
\end{proof}
It is clear that $T_{k}^\alpha:H^1_\#(\alpha,Y)^3\mapsto H^1_\#(\alpha,Y)^3$ is invertible when 
\begin{equation}
\label{HZ}
k\in \mathbb{C}\setminus Z^\alpha\;\text{ where}\; Z^\alpha=\{\frac{\tau_i(\alpha)-1/2}{\tau_i(\alpha)+1/2}\}_{\{-\frac{1}{2}<\tau_i(\alpha)<\frac{1}{2}\}}.
\end{equation}
So for $z=k^{-1}$, we have
\begin{equation}
\label{H34}
 (T_{k}^\alpha)^{-1}= z P^\alpha_1u+P^\alpha_2u+\sum_{-\frac{1}{2}<\tau_i(\alpha)<\frac{1}{2}}z[(1/2+\tau_i(\alpha))+z(1/2-\tau_i(\alpha))]P_{\tau_i(\alpha)}.
\end{equation}
For future reference we also introduce the set $S^\alpha$ of $z\in\mathbb{C}$ for which $T_k^\alpha$ is not invertible given by
\begin{equation}
    \label{H35}
    S^\alpha=\{\frac{\tau_i(\alpha)+1/2}{\tau_i(\alpha)-1/2}\}_{\{-1/2\leq \tau_i(\alpha)\leq 1/2\}}
\end{equation}
which also lies on the negative real axis. 
Collecting results, the spectral representation of the operator $ -\nabla\cdot (\mathbf{C}^1\chi_{D}(x)+\mathbf{C}^2\chi_{Y\setminus D}(x))\mathcal{E}$ on $H^1_\#(\alpha,Y)^3$ is given by
\begin{equation}
    \label{H36}
    -\nabla\cdot (\mathbf{C}^1\chi_{D}(x)+\mathbf{C}^2\chi_{Y\setminus D}(x))\mathcal{E}= -\mathcal{L}_\alpha T^\alpha_{k},
\end{equation}
in the sense of linear functionals over the space $H^1_\#(\alpha,Y)^3$ and recall that $-\mathcal{L}_\alpha$ is the Lam\'e operator associated with the bilinear form $\langle\cdot,\cdot\rangle$ defined on $H^1_\#(\alpha,Y)^3.$ This formulation is useful since it separates the effect of the contrast $k$ from the underlying geometry of the crystal. We note for future use that
$(-\mathcal{L}_\alpha)^{-1}$ is given by 
\begin{equation}
\label{B3}
(-\mathcal{L}_\alpha)^{-1}u(x)=-\int_Y\mathbf{G}^\alpha(x,y)u(y)\;dy.
\end{equation}
\section{Band Structure for Complex Coupling Constant}
\label{Band structure}
We set $\omega^2=\xi$ in $\eqref{I1}$. The operator representation is applied to write the Bloch eigenvalue problem as
\begin{equation}
\label{B1}
\begin{aligned}
 -\nabla\cdot (\mathbf{C}^1\chi_{D}(x)+\mathbf{C}^2\chi_{Y\setminus D}(x))\mathcal{E}(u) &= -\mathcal{L}_\alpha T^\alpha_{k} u=\xi \rho u\\
 (T^\alpha_{k})^{-1}(-\mathcal{L}_\alpha)^{-1}\rho u &= \frac{1}{\xi} u.
\end{aligned}
\end{equation}
We characterize the Bloch spectra by analysing the operator
\begin{equation}
\label{B2}
B^\alpha(k)=(T^\alpha_{k})^{-1}(-\mathcal{L}_\alpha)^{-1},
\end{equation}
and the operator given by the product $B^\alpha(k)\rho$.

It is shown in Theorem \ref{Dt5} that the operator  $B^\alpha(k):L^2_\#(\alpha,Y)^3\mapsto H^1_\#(\alpha, Y)^3$ is bounded for $k\notin Z^\alpha$. Thus the product $B^\alpha(k)\rho:L^2_\#(\alpha,Y)^3\mapsto H^1_\#(\alpha, Y)^3$ is also bounded. It follows from the compact embedding of $H^1_\#(\alpha, Y)^3$ into $L^2_\#(\alpha,Y)^3$ that $B^\alpha(k)\rho$ is compact on $L^2_\#(\alpha,Y)^3$ and therefore has a discrete spectrum $\{ \gamma_i(k,\alpha)\}_{i\in\mathbb{N}}$ with a possible accumulation point at $0$. The corresponding eigenspaces are finite dimensional and the eigenfunctions $p_i\in L^2_\#(\alpha,Y)^3$ satisfy
\begin{equation}
    \label{B4}
    [B^\alpha(k)\rho ]p_i(x)=\gamma_i(k,\alpha)p_i(x)\;\text{for}\;x\;\text{in}\;Y
\end{equation}
and also belong to $H^1_\#(\alpha, Y)^3$. Note further for $\gamma_i\neq 0$ that \eqref{B4} holds if and only if \eqref{B1} holds with $\xi_i(k,\alpha)=\gamma_i^{-1}(k,\alpha)$, and $-\mathcal{L}_\alpha T^\alpha_{k}  p_i=\rho \xi_i(k,\alpha)p_i$. Collecting results we have the following theorem
\begin{theorem}
\label{BT1}
Let $Z^\alpha$ denote the set of points on the negative real axis defined by \eqref{HZ}. Then the Bloch eigenvalue problem \eqref{I1} for the operator $-\nabla\cdot (\mathbf{C}(x)\mathcal{E} h(x))$ associated with the sesquilinear form \eqref{H6} can be extended for values of the coupling constant $k$ off the positive real axis into $\mathbb{C}\setminus Z^\alpha$, i.e., for each $\alpha\in Y^*$ the Block eigenvalues are of finite multiplicity and denoted by $\xi_j(k,\alpha)=\gamma_j^{-1}(k,\alpha), j\in\mathbb{N}$, and the band structure 
\begin{equation}
    \label{B5}
    \xi_j(k,\alpha)=\omega^2,\; j\in\mathbb{N}
\end{equation}
\end{theorem}
extends to complex coupling constants $k\in\mathbb{C}\setminus Z^\alpha$.

\section{Power series representation of Bloch eigenvalues for high contrast periodic media}
\label{Representation}
We set $\gamma=\xi^{-1}(k,\alpha)$ and analyse the spectral problem
\begin{equation}
\label{P1}
    [B^{\alpha}(k)\rho]u=\gamma(k,\alpha)u.
\end{equation}
We analyse the high contrast limit by developing a power series in $z=\frac{1}{k}$ about $z=0$ for the spectrum of the family of operators associated with \eqref{P1}.
\begin{align*}
B^{\alpha}(k)&:=(T_{k}^\alpha)^{-1}(-\mathcal{L}_\alpha)^{-1}\\
&=( z P^\alpha_1u+P^\alpha_2u+\sum_{-\frac{1}{2}<\tau_i(\alpha)<\frac{1}{2}}z[(1/2+\tau_i(\alpha))+z(1/2-\tau_i(\alpha))]P_{\tau_i(\alpha)})(-\mathcal{L}_\alpha)^{-1}\\
&=A^{\alpha}(z).
\end{align*}
We define the operator $A^\alpha(z)$ such that $A^\alpha(1/k)=B^\alpha(k)$ and the associated eigenvalues $\beta(1/k,\alpha)=\gamma(k,\alpha)$ and the spectral problem is $[A^\alpha(z)\rho] u=\beta(z,\alpha)u$ for $u\in L^2_\#(\alpha,Y)^3$.\\
From the above representation, it is easily seen that $A^\alpha(z)$  is self-adjoint for $k\in\mathbb{R}$ and is a family of bounded operators taking $L^2_\#(\alpha,Y)^3$ into itself.
\begin{lemma}
\label{PL1}
$A^\alpha(z)$ is holomorphic on $\Omega_0:=\mathbb{C}\setminus S^\alpha$, where $S^\alpha=\cup_{i\in\mathbb{N}} z_i$ is the collection of points $z_i=(1/2+\tau_i(\alpha))/(\tau_i(\alpha)-1/2)$ on the negative real axis associated with the eigenvalues $\{\tau_i(\alpha)\}_{i\in\mathbb{N}}$. The set $S^\alpha$ consists of poles of $A^\alpha(z)$ with accumulation points $z\in\{ -\frac{3\mu_1+\lambda_1}{\mu_1+\lambda_1}, -1, -\frac{\mu_1+\lambda_1}{3\mu_1+\lambda_1}\}$.
\end{lemma}
In section 8 we develop explicit $\alpha$ independent lower bounds $-1/2<\tau^-\leq\tau^-(\alpha)=\min_i\{\tau_i(\alpha)\}$, that hold for generic classes of inclusion domains $D$ and for every $\alpha\in Y^*$. The corresponding upper bound $z^+$ on $S^\alpha$ is written
\begin{equation}
\label{P2}
    \max_i\{z_i\}=\frac{\tau^-(\alpha)+1/2}{\tau^-(\alpha)-1/2}=z^*\leq z^+<0.
\end{equation}

Let $\beta_0^\alpha\in\sigma(A^\alpha(0)\,\rho)$ with spectral projection $P(0)$, and let $\Gamma$ be a closed contour in $\mathbb{C}$ enclosing $\beta_0^\alpha$ but no other element in $\sigma(A^\alpha(0)\,\rho).$ The spectral projection associated with $\beta^\alpha(z)\in\sigma(A^\alpha(z)\,\rho)$ for $\beta^\alpha(z)\in int(\Gamma)$ is denoted by $P(z).$ We write $M(z)=P(z)L^2_\#(\alpha, Y)^3$ and suppose for the moment that $\Gamma$ lies in the resolvent of $A^\alpha(z)\,\rho$ and $\dim(M(0)=\dim(M(z))=m.$ Now define $\hat{\beta}^\alpha(z)=\frac{1}{m}\operatorname{tr}(A^\alpha(z)\,\rho \,P(z)),$ the weighted mean of the eigenvalue group $\{\beta^\alpha_1(z),\dots\beta^\alpha_m(z)\}$ corresponding to $\beta^\alpha_0=\beta^\alpha_1(0)=\dots=\beta^\alpha_m(0).$ We write the weighted mean as 
\begin{equation}
    \label{P3}
\hat{\beta }^\alpha(z)=\beta^\alpha_0+\frac{1}{m}\operatorname{tr}[(A^\alpha(z)\,\rho-\beta^\alpha_0)P(z)].
\end{equation}
Since $A^\alpha(z)$ is analytic in a neighborhood of the origin we write 
\begin{equation}
    \label{P4}
    A^\alpha(z)=A^\alpha(0)+\sum_{n=1}^{\infty}z^n A^\alpha_n.
\end{equation}
The explicit form of the sequence $\{A^\alpha_n\}_{n\in\mathbb{N}}$ will be given later. Define the resolvent of the operator  $A^\alpha(z)\,\rho$ by
\[
R(\zeta,z)=(A^\alpha(z)\,\rho-\zeta)^{-1},
\]
and expanding successively in Neumann series and power series as in \cite{TKato3}  we obtain the resolvent as power series in $z$ with coefficents depending on $\zeta$ and $\rho$
\begin{equation}
    \label{P5}
    \begin{aligned}
    R(\zeta,z)&=R(\zeta,0)[I+(A^\alpha(z)\,\rho-A^\alpha(0)\rho)R(\zeta,0)]^{-1}\\
    &=R(\zeta,0)\sum_{p=0}^{\infty}[-(A^\alpha(z)\,\rho-A^\alpha(0)\,\rho)R(\zeta,0)]^p\\
    &=R(\zeta,0)+\sum_{n=1}^{\infty}z^nR_n(\zeta,\rho),
    \end{aligned}
\end{equation}
where
\[
R_n(\zeta,\rho)=\sum_{\stackrel{k_1+\dots+k_p=n}{k_j\geq 1}}(-1)^p R(\zeta,0)\, A^\alpha_{k_1}\,\rho\,R(\zeta,0) A^\alpha_{k_2}\,\rho \,R(\zeta,0)\cdots R(\zeta,0)A^\alpha_{k_p}\,\rho\, R(\zeta,0),
\]
where the sum is taken for all combinations of positive integers $p$ and $\{k_1,\ldots,k_p\}$ such that $1 \leq p\leq n$, $k_1+\cdots+k_p=n$.
Application of the contour integral formula for spectral projections delivers the spectral projection 
\begin{equation}
    \label{P6}
    \begin{aligned}
    P(z)&=-\frac{1}{2\pi i}\oint_\Gamma R(\zeta,z)\;d\zeta\\
    &=P(0)+\sum_{n=1}^{\infty}z^n P_n
    \end{aligned}
\end{equation}
where $P_n=-\frac{1}{2\pi i}\oint_\Gamma R_n(\zeta,\rho)\;d\zeta$. Now we develop the series for the weighted mean of the eigenvalue group associated with an eigenvalue $\beta^\alpha_0$ of geometric multiplicity $m$. start with 
\begin{equation}
\label{P7}
   (A^\alpha(z)\,\rho-\beta^\alpha_0)R(\zeta,z)=I+(\zeta-\beta^\alpha_0)R(\zeta,z)
\end{equation}
and we have 
\begin{equation}
\label{P8}
      (A^\alpha(z)\,\rho-\beta^\alpha_0)P(z)=-\frac{1}{2\pi i}\oint_\Gamma(\zeta-\beta^\alpha_0) R(\zeta,z)\;d\zeta,
\end{equation}
so from \eqref{P3}
\begin{equation}
\label{P9}
    \hat{\beta}(z)-\beta^\alpha_0=-\frac{1}{2m\pi i}\operatorname{tr}\oint_\Gamma(\zeta-\beta^\alpha_0) R(\zeta,z)\;d\zeta.
\end{equation}
Manipulation and integration by parts  as in \cite{TKato3}, Chap. 2, Sec. 2.2  yields
\begin{equation}
\label{P-10}
    \hat{\beta}(z)=\beta^\alpha_0+\sum_{n=1}^{\infty}z^n\beta^\alpha_n,
\end{equation}
where 
\begin{equation}
\label{P-11}
    \beta^\alpha_n=-\frac{1}{2m\pi i}\operatorname{tr}\sum_{\stackrel{k_1+\dots+k_p=n}{k_i\geq 1}}\frac{(-1)^p}{p}\oint_\Gamma A^\alpha_{k_1}\,\rho\,R(\zeta,0) A^\alpha_{k_2}\,\rho\,R(\zeta,0)\cdots A^\alpha_{k_p}\,\rho\,R(\zeta,0)\;d\zeta,
\end{equation}
as before the sum is taken for all combinations of positive integers $p$ and $\{k_1,\ldots,k_p\}$ such that $1 \leq p\leq n$, $k_1+\cdots+k_p=n$.
\section{Spectrum in the high contrast limit: Quasi-periodic case}
\label{Spectrum-quasiperiodic}
We now identify the limiting operator $A^\alpha(0)\,\rho$ when $\alpha\neq 0.$ Using the representation
\begin{equation}
\label{S1}
A^{\alpha}(z)= ( z P^\alpha_1+P^\alpha_2+\sum_{-\frac{1}{2}<\tau_i(\alpha)<\frac{1}{2}}z[(1/2+\tau_i(\alpha))+z(1/2-\tau_i(\alpha))]P_{\tau_i(\alpha)})(-\mathcal{L}_\alpha)^{-1},
\end{equation}
we see that
\begin{equation}
\label{S2}
    A^\alpha(0)\,\rho=P^\alpha_2(-\mathcal{L}_\alpha)^{-1}\,\rho.
\end{equation}
Denote the spectrum of $A^\alpha(0)\rho$ by $\sigma(A^\alpha(0)\rho).$ The following theorem provides the explicit characterization of $\sigma(A^\alpha(0)\rho)$. Let $-\mathcal{L}_D$ be the Lam\'e operator associated with the bilinear form $\langle\cdot,\cdot\rangle$ defined on $H^1_0(D)^3$. Recall the density is piece wise constant taking the value $\rho^1$ in $D$ and $\rho^2$ outside. Consider the Dirichlet eigenvalue problem 
\begin{equation}\label{dirichlet}
\mathcal{L}_Du=\rho^1\eta u \hbox{ for $\eta>0$ and $u\in H^1_0(D)^3$}. 
\end{equation}
The operator $-\mathcal{L}_D$ is invertable and the Dirichlet eigenvalues are given by the reciprocals of the discrete spectrum of $(-\mathcal{L}_D)^{-1}\rho^1$.
\begin{theorem}
\label{ST1}
 \[\sigma\left(A^\alpha(0)\rho\right)=\sigma\left((-\mathcal{L}_D)^{-1}\rho^1\right).\]
\end{theorem}
\begin{proof}
First we show that the eigenvalue problem
\[
P^\alpha_2(-\mathcal{L}_\alpha)^{-1}\rho^1 u=\eta u
\]
with $\eta\in\sigma(A^\alpha(0)\rho)$ and eigenfunction $u\in L^2_\#(\alpha, Y)^3$ is equivalent to finding $\eta$ and $u\in W^\alpha_2$ for which
\begin{equation}
\label{S3}
    (\rho u,v)=\eta\langle u,v\rangle,\;\text{for all}\;v\in W^\alpha_2.
\end{equation}
To see \eqref{S3}, note that we have $u=P^\alpha_2u$ and for $v\in H^1_\#(\alpha, Y)^3$,
\begin{equation}
    \label{S4}
    \langle P^\alpha_2(-\mathcal{L}_\alpha)^{-1}\rho u,v\rangle=\eta\langle u,v\rangle=\eta\langle P^\alpha_2u,v\rangle
\end{equation}
hence
\begin{equation}
\label{S5}
   \langle (-\mathcal{L}_\alpha)^{-1}\rho u,P^\alpha_2v\rangle=\eta\langle u,P^\alpha_2v\rangle.
\end{equation}
Since $\langle (-\mathcal{L}_\alpha)^{-1}\rho u,v\rangle=\int_{Y}\rho u\cdot\overline{v}\;dx=(\rho u,v)$ for any $u\in L^2_\#(\alpha,Y)^3)$ and $v\in H^1_\#(\alpha,Y)^3$, equation \eqref{S5} becomes
\begin{equation}
    \label{S6}
   (\rho u,P^\alpha_2v)=\eta\langle u,P^\alpha_2v\rangle,  
\end{equation}
since $P^\alpha_2$ is the projection of $ H^1_\#(\alpha,Y)^3$ onto $W^\alpha_2$, the equivalence follows.\\
We conclude by showing the set of eigenvalues for \eqref{S3} is given by $\sigma((-\mathcal{L}_D)^{-1}\rho^1)$. 
Let $\mathcal{R}$ be the space of rigid motions on $Y\setminus D$ and note that the kernel of the symmetric gradient on $Y\setminus D$ is $\mathcal{R}$. Define $\tilde{H}^1_0(D)^3$ to be the subspace of functions $H^1_0(D)^3$  extended by zero into $Y\setminus D$.
Since $W^\alpha_2\cap\mathcal{R}=0$ we see that $W^\alpha_2=\tilde{H}^1_0(D)^3$. Since $P^\alpha_2v$ is supported in $D$ \eqref{S3} is
\begin{equation}
   \label{S7}
   \eta^{-1}\int_D \rho^1 u\cdot\overline{P^\alpha_2v}=\int_{D}\mathbf{C}^1\mathcal{E}(u):\overline{\mathcal{E}(P^\alpha_2v)}\; dx.
\end{equation}
Now since $P^\alpha_2:H^1_\#(\alpha,Y)^3\mapsto W^\alpha_2=\tilde{H}^1_0(D)^3$ is onto, it follows that $\eta^{-1}$ is a eigenvalue of \eqref{dirichlet}.
\end{proof}

\section{Spectrum in the high contrast limit: Periodic case}
\label{Spectrum-periodic}
For the periodic case, $P^0_2$ is the projection onto $W^0_2$. The limiting operator is written 
\begin{equation}
\label{SS1}
    A^0(0)\,\rho=P^0_2(-\mathcal{L}_0)^{-1}\,\rho.
\end{equation}
Here the operator $(-\mathcal{L}_0)^{-1}$ is compact and self-adjoint on $L^2_\#(0,Y)^3$ and given by
\begin{equation}
\label{SS2}
    (-\mathcal{L}_0)^{-1}u(x)=-\int_Y \mathbf{G}^0(x,y)u(y)\;dy.
\end{equation}
Denote the spectrum of $A^0(0)\rho$ by $\sigma(A^0(0)\rho)$. To characterize this spectrum we introduce the effective mass tensor
\begin{equation}
\label{Massmatrix}
    M(\nu)=I\int_Y\rho(x)\;dx-\nu\sum_{j\in\mathbb{N}}\frac{\int_D\rho^1 \overline{\psi}_j\;dx\otimes\int_D\rho^1\psi_j\;dx}{\nu-\delta^*_j},
\end{equation}  
where $I$ is the $3\times 3$ identity and $\{\delta^*_j\}_{j\in\mathbb{N}}$ are the Dirichlet eigenvalues of $(-\mathcal{L}_D) u=\rho^1\eta u$, $u\in H_0^1(D)^3$ associated with eigenfunctions $\psi_j$ for which $\int_D\rho^1\psi_j\;dx\neq 0$.
\begin{remark}\label{mass matrix}
The effective mass tensor $M(\nu)$ is precisely the effective mass tensor of the high contrast elastic metamaterial \cite{Vondrejc} and \cite{Comi} associated with a sub-wavelength periodic lattice. Matrices of a similar type corresponding to the effective magnetic permeability tensor for photonic metamaterials with artificial magnetism are identified in \cite{bou1}, \cite{bou2}, and also appear in the homogenization theory of high contrast porous media \cite{Zhikov}.
\end{remark}

Next we introduce the sequence of numbers $\{\nu_j\}_{j\in\mathbb{N}}$ given by the positive roots $\nu$ of  the determinant of the effective mass matrix 
\begin{equation}
\label{SS3}
   det\left\{ M(\nu)\right\}=0
\end{equation}  
The following theorem provides the explicit characterization of $\sigma(A^0(0)\rho)$.
\begin{theorem}
\label{SSt6.1}
Let $\{(\delta'^{-1}_j\}_{j\in\mathbb{N}}$ denote the collection of eigenvalues for  $(-\mathcal{L}_D)^{-1}\rho^1$ associated with eigenfunctions $\psi_j$ of \eqref{S2} for which $\int_D \rho^1\psi_j=0$. Then $\sigma(A^0(0)\rho)=\{\delta'^{-1}_j\}_{j\in\mathbb{N}}\cup\{\nu^{-1}_j\}_{j\in\mathbb{N}}$.
\end{theorem}

\begin{proof} We argue as in the previous section to find that the eigenvalue problem 
\[
P^0_2(-\mathcal{L}_0)^{-1}\rho=\eta u
\]
with $\eta\in\sigma(A^0(0)\rho)$ and eigenfunction $u\in L^2_\#(0,Y)^3$ is equivalent to finding $\eta$ and $u\in W^0_2$ for which 
\begin{equation}
\label{SS4}
     (\rho u,v)=\eta\langle u,v\rangle,\;\text{for all}\;v\in W^0_2.
\end{equation}
We show that the eigenvalues $\eta_j$ for \eqref{SS4} are given by the alternative:
\begin{equation}\label{alt}
\eta_j=(\delta'_j)^{-1},  \hbox{  or  } \eta_j=(\nu_j)^{-1}.
\end{equation}
From \eqref{H9} we have the dichotomy: $\int_D\tilde{u}\;dx=0$ and $u=\tilde{u}\in\tilde{H}^1_0(D)^3$ or $\int_D\tilde{u}\;dx\neq0$ and $u=\tilde{u}-\gamma1_Y$ with $\gamma=\langle\rho\rangle^{-1}\int_D\rho^1\tilde{u}\;dx$. For the first case that the eigenfunction belongs to $\tilde{H}^1_0(D)^3$ and for $v\in W^0_2$ given by
\begin{equation}
\label{SS5}
    v=\tilde{v}-\langle\rho\rangle^{-1}\big(\int_D\rho^1\tilde{v}\;dx\big)1_Y \;\text{for}\; \tilde{v}\in\tilde{H}^1_0(D)^3
\end{equation}
the problem \eqref{SS4} becomes
\begin{equation}
\label{SS6}
     \int_D \rho^1 u\cdot\overline{\tilde{v}}=\eta\int_{D}\mathbf{C}^1\mathcal{E}(u):\overline{\mathcal{E}(\tilde{v})}\; dx,\;\text{for all}\;\tilde{v}\in\tilde{H}^1_0(D)^3,
\end{equation}
and we conclude that $\tilde{u}$ is a Dirichlet eigenfunction with zero average over $D$ so $\eta\in \{\delta'^{-1}_j\}_{j\in\mathbb{N}}$. For the second case, we have $u\in W^0_2$ and again
\begin{equation}
\label{SS7}
     \int_D \rho^1 u\cdot\overline{\tilde{v}}=\eta\int_{D}\mathbf{C}^1\mathcal{E}(u):\overline{\mathcal{E}(\tilde{v})}\; dx,\;\text{for all}\;\tilde{v}\in\tilde{H}^1_0(D)^3.
\end{equation}
Writing $u=\tilde{u}-\gamma1_Y$ and integration by parts in \eqref{SS7} shows that $\tilde{u}\in\tilde{H}^1_0(D)^3$ is the solution of
\begin{equation}
    \label{SS8}
    \mathcal{L}_D\tilde{u}+\nu\rho^1\tilde{u}=\nu\rho^1\gamma\;\text{for}\;x\in D.
\end{equation}
Since $\tilde{u}\in\tilde{H}^1_0(D)^3$ we can write
\begin{equation}
\label{SS9}
    \tilde{u}=\sum_{j=1}^{\infty}c_j\psi_j
\end{equation}
where, $\psi_j$ are the Dirichlet eigenfunctions of \eqref{dirichlet} associated with eigenvalue $\delta_j$ extended by zero to $Y$. Substitution of \eqref{SS9} into \eqref{SS8} yields
\begin{equation}
\label{SS10}
   \sum_{j=1}^{\infty}(-\delta_j\rho^1+\nu\rho^1)c_j\psi_j=\nu\rho^1\gamma.
\end{equation}
Multiplying both sides of \eqref{SS10} by $\overline{\psi}_k$ over $D$ and  $\int_D\rho^1\psi_i\cdot\overline{\psi}_j dx=\delta_{ij}$ shows that $\tilde{u}$ is given by
\begin{equation}
\label{SS11}
    \tilde{u}=\nu\sum_{k\in\mathbb{N}}\frac{\gamma\cdot\,\int_D\rho^1\overline{\psi}_k}{\nu-\delta^*_k}\psi_k,
\end{equation}
where $\delta^*_k$ correspond to Dirichlet eigenvalues associated with eigenfunctions for which  $\int_D\rho^1\psi_k\;dx\neq0$. Hence
\begin{equation}
\label{SS112}
{u}=\nu\sum_{k\in\mathbb{N}}\frac{\gamma\cdot\,\int_D\rho^1\overline{\psi}_k}{\nu-\delta^*_k}\psi_k-\gamma.
\end{equation}
To find $\nu$, we multiply both sides of \eqref{SS112} by $\rho(x)$ and integrate both sides over $Y$ to recover the identity
\begin{equation}
\label{SS113}
  \operatorname{det} \left\{ M(\nu)\right\}=0.
\end{equation}
Hence we conclude that $\eta\in\{\nu^{-1}_i \}$ and the proof is complete.
\end{proof}
We conclude with a varational characterization of the eigenvalues. Let $\{\delta_j\}_{j\in\mathbb{N}}$ denote all the Dirichlet eigenvalues of $(-\mathcal{L}_D) u=\rho^1\delta u$,  $u\in H^1_0(D)^3$, i.e., $\{\delta_j\}_{j\in\mathbb{N}}=\{\delta^{'}_j\}_{j\in\mathbb{N}}\cup \{\delta_j^\ast\}_{j\in\mathbb{N}}$.
Then one readily obtains the min-max characterizations of $\eta^{-1}_j$ and $\delta_j$ given by
\begin{lemma}
\label{minmax}
\begin{equation}
    \label{deltaminmax}
    \delta_j=\min_{S^j\subseteq H^1_0(Y)^3}\Big\{\max_{0\neq u\in S^j}\frac{\int_Y\mathbf{C}^1\mathcal{E}(u):\mathcal{E}(\overline{u})dx}{(\rho u,u)}\Big\}
\end{equation}
\begin{equation}
\label{numinmax}
{\eta^{-1}_j}=\min_{S^j\subseteq H_0^1(D)^3}\Big\{\max_{0\neq u\in S^j}\frac{\int_Y\mathbf{C}^1\mathcal{E}(u):\mathcal{E}(\overline{u})dx}{(\rho u,u)-\langle\rho\rangle^{-1}|\int_Y \rho u|^2}\Big\}
\end{equation}
\end{lemma}
\begin{proof}
The min-max formulation for Dirichlet eigenvalues \eqref{deltaminmax} is standard \cite{TKato3}. The second identity follows from the standard min-max formulation  on noting that for $u\in W_2^0$ that $u=\tilde{u}-{\langle \rho\rangle^{-1}}\int_Y\rho\tilde{u}\,dx$ for $\tilde{u}\in H^1_0(D)^3$ and $0\leq (\rho u,u)=(\rho \tilde{u},\tilde{u})-\langle\rho\rangle^{-1}|\int_Y \rho \tilde{u}|^2$, when $\int_Y\rho\tilde{u}\,dx\not=0$. 
\end{proof}

\section{Radius of convergence and convergence rates}
\label{Radius of conv}
Fix an inclusion geometry specified by the domain $D$. Suppose that $\alpha\in Y^*$ and $\alpha\neq 0$. Recall from Theorem \ref{ST1} that the spectrum of $A^\alpha(0)\rho$ is $\sigma((-\mathcal{L}_D)^{-1}\rho^1)$. Take $\Gamma$ to be a closed contour in $\mathbb{C}$ containing an eigenvalue $\beta^\alpha_j(0)$ in $\sigma((-\mathcal{L}_D)^{-1}\rho^1)$ but no other element of $\sigma((-\mathcal{L}_D)^{-1}\rho^1)$.
\begin{figure}[h]
\centering
\begin{tikzpicture}[xscale=0.70,yscale=0.70]
\draw [-,thick] (-6,0) -- (6,0);
\draw [<->,thick] (0,0) -- (1.5,0);
\node [above] at (1,0) {$d$};
\draw [<->,thick] (1.5,0) -- (3,0);
\node [above] at (2,0) {$d$};
\draw (-4,0.2) -- (-4.0, -0.2);
\node [below] at (-4,0) {$\hat{\beta}^\alpha_{j-1}(0)$};
\draw (0,0.2) -- (0, -0.2);
\node [below] at (0,0) {$\beta^\alpha_j(0)$};
\draw (3,0.2) -- (3, -0.2);
\node [below] at (3,0) {$\check{\beta}^\alpha_{j+1}(0)$};
\draw (0,0) circle [radius=1.5];
\node [right] at (1.1,1.1) {$\Gamma$};
\end{tikzpicture} 
\caption{$\Gamma$}
\end{figure} Define $d$ to be the distance between $\Gamma$ and $\sigma((-\mathcal{L}_D)^{-1})\rho^1)$, i.e.,
\begin{equation}
\label{R1}
d=\operatorname{dist}(\Gamma,\sigma((-\mathcal{L}_D)^{-1}\rho^1)=\inf_{\xi\in\Gamma}\{\operatorname{dist}(\Gamma,\sigma((-\mathcal{L}_D)^{-1}\rho^1).
\}
\end{equation}
The component of the spectrum of $A^\alpha(0)\rho$ inside $\Gamma$ is precisely $\beta^\alpha_j(0)$ and we denote this by $\Sigma'(0)$. The part of the spectrum of $A^\alpha(0)\rho$ in the domain exterior to $\Gamma$ is denoted by $\Sigma''(0)$ and $\Sigma''(0)=\sigma((-\mathcal{L}_D^{-1})\rho^1)\setminus{\beta^\alpha_j(0)}$. The invariant subspace of $A^\alpha(0)\rho$ associated with $\Sigma'(0)$ is denoted by $M'(0)$ with $M'(0)=P(0)L^2_\#(\alpha,Y)^3$.

Suppose the lowest quasi-periodic resonance eigenvalue for the domain $D$ lies inside $-1/2<\tau^-(\alpha)$. It is noted that in the sequel a large and generic class of domains are identified for which there exists $\tau^-$, independent of $\alpha\in Y^\ast$ such that $-1/2<\tau^-\leq\tau^-(\alpha)$. The corresponding upper bound on the set $z\in S^\alpha$ for which $A^\alpha(z)\rho$ is not invertible is given by 
\begin{equation}
    \label{R2}
    z^*=\frac{\tau^-(\alpha)+1/2}{\tau^-(\alpha)-1/2}<0,
\end{equation}
see \ref{P2}. Now set
\begin{equation}
    \label{R3}
    r^*=\frac{\mu_1|\alpha|^2d|z^*|}{\frac{\|\rho\|_{L^\infty(Y)^3}}{1/2-\tau^-(\alpha)}+\mu_1|\alpha|^2d}.
\end{equation}
\begin{theorem}
\label{Rt1}
Separation of spectra and radius of convergence for $\alpha\in Y^*,\alpha\neq0$. The following properties hold for inclusions with domains $D$ that satisfy \eqref{R2}:
\begin{enumerate}
    \item If $|z|<r^*$ then $\Gamma$ lies in the resolvent of both $A^\alpha(0)\rho$ and $A^\alpha(z)\rho$ and thus separates the spectrum of $A^\alpha(z)\rho$ into two parts given by the components of spectrum of $A^\alpha(z)\rho$ inside $\Gamma$ denoted by $\Sigma'(z)$ and components exterior to $\Gamma$ denoted by $\Sigma''(z)$. The invariant subspace of subspace of $A^\alpha(z)\rho$ associated with $\Sigma'(z)$ is denoted by $M'(z)$ with $M'(z)=P(z)L^2_\#(\alpha,Y)^3$.
    
    \item The projection $P(z)$ is holomorhic for $|z|<r^*$ and $P(z)$ is given by 
    \begin{equation}
        \label{R4}
        P(z)=-\frac{1}{2\pi i}\oint_\Gamma R(\zeta,z)\;d\zeta.
    \end{equation}
    
    \item The spaces $M'(z)$ and $M'(0)$ are isomorphic for $|z|<r^*$.
    \item The power series \eqref{P-10} converges uniformly for $z\in\mathbb{C}$ inside $|z|<r^*$.
\end{enumerate}
\end{theorem}
Suppose now $\alpha=0$. Recall from Theorem \ref{SSt6.1} that the limit spectrum for $A^\alpha(0)\rho$ is $\sigma(A^\alpha(0)\rho)=\{\delta'^{-1}_j\}_{j\in\mathbb{N}}\cup\{\nu^{-1}_j\}_{j\in\mathbb{N}}$. For this case take $\Gamma$ to be the closed contour in $\mathbb{C}$ containing an eigenvalue $\beta^0_j(0)$ in $\sigma(A^0(0)\rho)$ but no other element of $\sigma(A^0(0)\rho)$ and define
\begin{equation}
    \label{R5}
    d=\inf_{\zeta\in\Gamma}\{\operatorname{dist}(\zeta,\sigma(A^0(0)\rho)\}.
\end{equation}
Suppose the lowest quasi-periodic resonance eigenvalue for the domain $D$ lies inside $-1/2<\tau^-(0)<0$ and the corresponding upper bound on $S^0$ is given by
\begin{equation}
\label{R6}
    z^*=\frac{\tau^-(0)+1/2}{\tau^-(0)-1/2}<0.
\end{equation}
Set 
\begin{equation}
    \label{R7}
   r^*=\frac{4\pi^2\mu_1d|z^*|}{\frac{\|\rho\|_{L^\infty(Y)^3}}{1/2-\tau^-(0)}+4\pi^2\mu_1d}.
\end{equation}
\begin{theorem}
\label{Rt2}
Separation of spectra and radius of convergence for $\alpha=0$.\\
The following properties hold for inclusions with domains $D$ that satisfy \eqref{R6}:
\begin{enumerate}
    \item  If $|z|<r^*$ then $\Gamma$ lies in the resolvent of both $A^0(0)\rho$ and $A^0(z)\rho$ and thus separates the spectrum of $A^0(z)\rho$ into two parts given by the components of spectrum of $A^0(z)\rho$ inside $\Gamma$ denoted by $\Sigma'(z)$ and components exterior to $\Gamma$ denoted by $\Sigma''(z)$. The invariant subspace of subspace of $A^0(z)\rho$ associated with $\Sigma'(z)$ is denoted by $M'(z)$ with $M'(z)=P(z)L^2_\#(0,Y)^3)$.
    
    \item The projection $P(z)$ is holomorhic for $|z|<r^*$ and $P(z)$ is given by 
    \begin{equation}
        \label{R8}
        P(z)=-\frac{1}{2\pi i}\oint_\Gamma R(\zeta,z)\;d\zeta.
    \end{equation}
      
    \item The spaces $M'(z)$ and $M'(0)$ are isomorphic for $|z|<r^*$.
    \item The power series \eqref{P-10}  converges uniformly for $z\in\mathbb{C}$ inside $|z|<r^*$.
\end{enumerate}
\end{theorem}
Next we provide an explicit representation of the integral operators appearing in the series expansion for the eigenvalue group.
\begin{theorem}
\label{Rt3}
Representation of integral operators in the series expansion of eigenvalues. Let $P^\alpha_3$ be the projection onto the orthogonal complement of $W^\alpha_1\oplus W^\alpha_2$  and let $\tilde{I}$ denote the identity on $L^2(\partial D)^3$, then the explicit representation for the operators $A_n^\alpha$ in the expansion \eqref{P-10},\eqref{P-11} is given by
\begin{equation}
\label{R9}
\begin{aligned}
    A^\alpha_1&=[\mathcal{S}^\alpha_D(\mathcal{K}^{-\alpha}_{D})^*+\frac{1}{2}\tilde{I})(\mathcal{S}^\alpha_D)^{-1})^{-1}P^\alpha_3+P^\alpha_1](-\mathcal{L}_\alpha)^{-1}\;\text{and}\\
     A^\alpha_n&=\mathcal{S}^\alpha_D((\mathcal{K}^{-\alpha}_{D})^*+\frac{1}{2}\tilde{I})^{-1}(\mathcal{S}^\alpha_D)^{-1})^{-1}[\mathcal{S}^\alpha_D((\mathcal{K}^{-\alpha}_{D})^*-\frac{1}{2}\tilde{I})\mathcal{S}^\alpha_D((\mathcal{K}^{-\alpha}_{D})^*+\frac{1}{2}\tilde{I})^{-1}(\mathcal{S}^\alpha_D)^{-1}]^{n-1}P^\alpha_3(-\mathcal{L}_\alpha)^{-1}.
    \end{aligned}
\end{equation}
\end{theorem}
\begin{theorem}
\label{Rt4}
Error estimates for the eigenvalue expansion.
\begin{enumerate}
\item Let $\alpha\neq 0$, and suppose $D, z^*,$ and $r^*$ are as in Theorem \ref{Rt1}. Then the following error estimate for the series \eqref{P-10} holds for $|z|<r^*$:
\begin{equation}
    \label{R10}
    \Big|\hat{\beta}^\alpha(z)-\sum_{n=0}^{p}z^n\beta^\alpha_n\Big|\leq \frac{d|z|^{p+1}}{(r^*)^p(r^*-|z|)}.
\end{equation}
\item
Let $\alpha= 0$, and suppose $D, z^*,$ and $r^*$ are as in Theorem \ref{Rt2}. Then the following error estimate for the series \eqref{P-10} holds for $|z|<r^*$:
\begin{equation}
    \label{R11}
    \Big|\hat{\beta}^0(z)-\sum_{n=0}^{p}z^n\beta^0_n\Big|\leq \frac{d|z|^{p+1}}{(r^*)^p(r^*-|z|)}.
\end{equation}
\end{enumerate}
\end{theorem}
We summarize results in the following theorem.
\begin{theorem}
\label{Rt5}
The Bloch eigenvalue problem \eqref{I1} is defined for the coupling constant $k$ extended into the complex plane and the operator $u\mapsto\nabla\cdot(\mathbf{C}^1\chi_{D}+\mathbf{C}^2\chi_{Y\setminus D})\mathcal{E}(u)$ with domain $H_\#^1(\alpha, Y)^3$ is holomorphic for $k\in\mathbb{C}\setminus Z^\alpha$. The associated Bloch spectra is given by the eigenvalues $\xi_j(k,\alpha)=(\beta_j^\alpha(1/k))^{-1}$, for $j\in\mathbb{N}$. For $\alpha\in Y^*$ fixed, the eigenvalues are of finite multiplicity. Moreover for each $j$ and $\alpha\in Y^*$ the eigenvalue group is analytic within a neighborhood of infinity containing the disk $|k|>(r^*)^{-1}$ where $r^*$ is given by \eqref{R3} for $\alpha\neq 0$ and by \eqref{R7} for $\alpha=0$. When $\beta_j^\alpha(0)$ is simple these conditions are sufficient for the separation of spectral branches of the dispersion relation for fixed quasi-momentum within a neighborhood of infinity containing the disk $|k|>(r^*)^{-1}$.
\end{theorem}
The proofs of Theorems \ref{Rt1}, \ref{Rt2} and \ref{Rt4} are given in section \ref{sec-derivation}. The proofs of Theorem \ref{Rt3} is given in section \ref{sec-layerpotential}.

\section{Bounds on Quasi-static Resonance Spectra for Periodic Scatters of General Shape}
\label{Separation of spectra}
In this section we identify an explicit condition on the inclusion geometry that guarantees a lower bound $\tau^-$ on the quasi-periodic spectra that holds uniformly for $\alpha\in Y^*$. 
$$-\frac{1}{2}<\tau^-\leq \tau^-(\alpha)=\min_i\{\tau_i\}\leq\frac{1}{2},$$
This provides a lower bound on the structural spectra $\cup_{\alpha\in Y^\ast}\{\tau_i(\alpha)\}_{i=1}^{\infty}$ that is strictly greater than $-1/2$.

To begin, note if $(\tau,w)$ is an eigenpair of $T|_{W_3}$ and $v\in H_\#^1(\alpha, Y)^3$ then,
\begin{equation}
\label{G1}
    \frac{1}{2}\int_{Y\setminus D}\mathbf{C}^1\mathcal{E}(w):\overline{\mathcal{E}(v)}\; dx -\frac{1}{2}\int_{D}\mathbf{C}^1\mathcal{E}(w):\overline{\mathcal{E}(v)}\;dx=\tau\int_{Y}\mathbf{C}^1\mathcal{E}(w):\overline{\mathcal{E}(v)}\; dx.
\end{equation}
Adding $ \frac{1}{2}\int_{Y}\mathbf{C}^1\mathcal{E}(w):\overline{\mathcal{E}(v)}\; dx$ to both sides yields
\begin{equation}
\label{G2}
    \int_{Y\setminus D}\mathbf{C}^1\mathcal{E}(w):\overline{\mathcal{E}(v)}\; dx=(\tau+\frac{1}{2})\int_{Y}\mathbf{C}^1\mathcal{E}(w):\overline{\mathcal{E}(v)}\; dx.
\end{equation}
The lower bound is obtained by  showing  that there exists a $p>0$ such that $\tau_i+\frac{1}{2}\geq p$ independent of $i\in\mathbb{N}$ and $\alpha\in Y^*$. 
\begin{theorem}
\label{Gt1}
Let $\tau^-(\alpha)$ be the lowest eigenvalue of $T$ in $W_3^\alpha\subset H_\#^1(\alpha, Y)^3$. Suppose there is a $\theta> 0$ independent of $\alpha\in Y^\ast$ such that for all $v\in W^\alpha_3$ we have 
\begin{equation}
\label{G3}
     \int_{Y\setminus D}\mathbf{C}^1\mathcal{E}(v):\overline{\mathcal{E}(v)}\; dx\geq \theta  \int_{D}\mathbf{C}^1\mathcal{E}(v):\overline{\mathcal{E}(v)}\; dx.
\end{equation}
Let $p=\min\{\frac{1}{2},\frac{\theta}{2}\}.$ Then $\tau^-(\alpha)+\frac{1}{2}\geq p$ for all $\alpha\in Y^*.$
\end{theorem}
\begin{proof}
We prove by contradiction, so in addition to \eqref{G3}, we suppose that $\tau^-(\alpha)+\frac{1}{2}< p$ for some $\alpha\in Y^*$. Then $\tau^-(\alpha)+\frac{1}{2}<\frac{1}{2}$ and $\tau^-(\alpha)+\frac{1}{2}<\frac{\theta}{2}.$ Let $u^-$ be the eigenvector of $T$ with eigenvalue $\tau^-(\alpha),$ normalized so that $\Vert u^-\Vert=1$. Then we obtain
\begin{equation}
\label{G4}
     \int_{Y\setminus D}\mathbf{C}^1\mathcal{E}(u^-):\overline{\mathcal{E}(u^-)}\; dx<\frac{1}{2}
\end{equation}
and
\begin{equation}
\label{G5}
    \frac{\theta}{2}>  \int_{Y\setminus D}\mathbf{C}^1\mathcal{E}(u^-):\overline{\mathcal{E}(u^-)}\; dx\geq\theta  \int_{D}\mathbf{C}^1\mathcal{E}(u^-):\overline{\mathcal{E}(u^-)}\; dx.
\end{equation}
This gives
\begin{equation}
\label{G6}
    \int_{D}\mathbf{C}^1\mathcal{E}(u^-):\overline{\mathcal{E}(u^-)}\; dx<\frac{1}{2}.
\end{equation}
Inequalities \eqref{G4} and \eqref{G6} yield
\[
\int_{Y}\mathbf{C}^1\mathcal{E}(u^-):\overline{\mathcal{E}(u^-)}\; dx=\|u^-\|^2<1.
\]
This is a contradiction since $\|u^-\|=1$.
\end{proof}
Clearly the parameter $\theta$ is a geometric descriptor for $D$. We define the class of inclusion configurations for which the structural spectra is bounded strictly above $-1/2$.
\begin{definition}\label{Ptheta}
The class of periodic distributions of inclusions for which Theorem \ref{Gt1} holds for a fixed positive value of $\theta$ is denoted by $P_\theta$. 
The structural spectra $\cup_{\alpha\in Y^\ast}\{\tau_i(\alpha)\}_{i=1}^{\infty}$ for this class is bounded above $-1/2$ so $z^\ast<0$ is uniformly bounded away from zero for $\alpha\in Y^*$.
\end{definition}
With this definition we have the corollary given by:
\begin{corollary}
\label{Gc2}
For every inclusion domain $D$ belonging to $P_\theta$ Theorems \ref{Rt1} through \ref{Rt5} hold with $z^*$ replaced with $z^+_\theta$ given by
\begin{equation}
\label{G7}
z^+_\theta=\frac{\tau^-+1/2}{\tau^--1/2}<0,
\end{equation}
where $\tau^-=\min\{\frac{1}{2},\frac{\theta}{2}\}-\frac{1}{2}$.
\end{corollary}
Now we introduce a wide class of inclusion shapes belonging to $P_\theta$ for a given $\theta>0$. Consider an inclusion domain $D$ with smooth $C^\infty$ boundary. Suppose we can surround the inclusion with a security layer $R$ of given thickness such that their union ${D}\cup R=D'$ is contained inside the interior of $Y$. We show next that there is a $\theta>0$ that will depend on $R$ and $D$ but be independent of $\alpha\in Y^\ast$. Given $f'\in H^1(R)^3$ there is a bounded linear extension operator $E:H^1(R)^3\mapsto H^1(D')^3$ such that $f=E(f')$ satisfies $f(x)=f'(x)$ for $x\in R$ see \cite{Necas}. Hence there is a positive constant $C_e$ depending only on $D$ and $R$ such that
\begin{equation}
\label{G1154}
    \begin{aligned}
    \|E(f')\|_{H^1(D)^3}^2
    &\leq C_e\|f'\|^{2}_{H^1(R)^3}.
\end{aligned}
\end{equation}
The space of rigid body motions on $R$ is written $\mathcal{R}=\{u(x)=Q x+c\; ;\,\,x\in R,\,\,Q\in SO^3, c\in\mathbb{R}^3\}$ and its projection with respect to the $L^2(R)^3$ norm is written $\mathbb{P}_{\mathcal{R}}$. Now choose $u\in W_3^\alpha$
and consider $u-\mathbb{P}_\mathcal{R}u$ restricted to $R$, and we have the inequalities:
\begin{equation}
\label{G1155}
    \begin{aligned}
    \frac{1}{\beta}\int_D\,\mathbf{C}^1\mathcal{E} E(u-\mathbb{P}_\mathcal{R} u):\overline{\mathcal{E}E(u-\mathbb{P}_\mathcal{R} u)}\,dy&\leq \|E(u-\mathbb{P}_{\mathcal{R}} u)\|_{H^1(D)^3}^2
    \leq C_e\|u-\mathbb{P}_\mathcal{R} u\|^{2}_{H^1(R)^3}.
\end{aligned}
\end{equation}
To obtain the first inequality we use $\gamma<\mathbf{C}^1< \beta$ in the sense of quadratic forms, see \eqref{ellip}.  On applying Korn and Poincare inequalities to the right hand side (see, e.g., \cite{Duvaut} pg. 117) delivers positive  constants $K$ and $C$ independent of $\alpha\in Y^\ast$ such that
\begin{equation}
\label{G156}
    \begin{aligned}
    \|u-\mathbb{P}_\mathcal{R}u\|_{H^1(R)^3}^2
    &\leq K\left(\int_{R } |u-\mathbb{P}_\mathcal{R}u|^2\,dw+\int_{R}\,\mathcal{E}(u):\overline{\mathcal{E}(u)}\,dx\right),
\end{aligned}
\end{equation}
and 
\begin{equation}
\label{G157}
    \begin{aligned}
    \|u-\mathbb{P}_\mathcal{R}u\|_{L^2(R)^3}^2
    &\leq C\int_{R}\,\mathcal{E}(u):\overline{\mathcal{E}(u)}\,dx,
\end{aligned}
\end{equation}
Noting that $u\in W_3^\alpha$ satisfies $\mathcal{L} u=0$ in $D$ we see that it is a minimizer of the elastic energy on $D$ for boundary data $u-\mathbb{P}_\mathcal{R}u$ on $\partial D$  hence
\begin{equation}
\label{G158}
    \begin{aligned}
    \frac{1}{\beta}\int_D\,\mathbf{C}^1\mathcal{E} (u):\overline{\mathcal{E}(u)}\,dy&
    \leq C_e\|u-\mathbb{P}_\mathcal{R} u\|^{2}_{H^1(R)^3}.
\end{aligned}
\end{equation}
Application of \eqref{G156}, \eqref{G157}, \eqref{G158}, and \eqref{ellip} gives
\begin{equation}
\label{G159}
    \begin{aligned}
    \theta\int_D\,\mathbf{C}^1\mathcal{E} (u):\overline{\mathcal{E}(u)}\,dy&
    \leq \int_R\,\mathbf{C}^1\mathcal{E} (u):\overline{\mathcal{E}(u)}\,dy,
\end{aligned}
\end{equation}
with $\theta$ independent of $\alpha\in Y^\ast$ and $\theta=\frac{\gamma}{\beta KC_e(1+C)}$ and it follows that
\begin{equation}
\label{G160}
    \begin{aligned}
    \theta\int_D\,\mathbf{C}^1\mathcal{E} (u):\overline{\mathcal{E}(u)}\,dy&
    \leq \int_{Y\setminus D}\,\mathbf{C}^1\mathcal{E}(u):\overline{\mathcal{E}(u)}\,dy,
\end{aligned}
\end{equation}

\section{Radius of Convergence and rates of convergence for Dispersions of Spherical inclusions}
\label{Example}

We now provide an example where the radius of convergence and separation of spectra given by Theorems \ref{Rt1} and \ref{Rt2} is determined explicitly by the radii of each inclusion, the minimum distance seperating each inclusion and the Dirichlet spectra of the inclusions.
We display this for crystals in $\mathbb{R}^3$ with period cell containing a of spherical inclusion $D$  of radius $a$ surrounded by a buffer shell $R$ of inner radius $a$ and outer radius $b$.
The spherical inclusion can be located any where inside the unit cell provided that the buffer shell is also interior to the unit cell. This is a  specific example of the $P_\theta$ type geometry, see Figure \ref{plane2} introduced in the work of \cite{Bruno}.
 
\begin{figure}[h]
\centering
\begin{tikzpicture}[xscale=0.6,yscale=0.6]
\draw [thick] (-2,-2) rectangle (3,3);
\draw [fill=orange,thick] (0.95,0.95) circle [radius=1.0];
\draw (0.95,0.95) circle [radius=2.0];

\end{tikzpicture} 
\caption{\bf The shaded region is the inclusion
of radius $a$ surrounded by a shell of thickness $a$ and outer radius $2a$.}
 \label{plane2}
\end{figure}
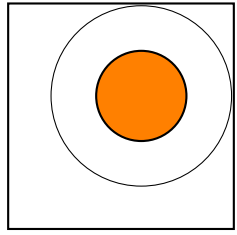

Let $q$ denote the ratio between inner and outer radius and to illustrate the ideas we choose $q=0.5$. Table 2 of
\cite{Bruno} delivers $\theta=1/2$ consequently $\tau^-=-0.25$ and together with \eqref{G7}  gives $|z^*|\leq\frac{\tau^-+1/2}{1/2-\tau^-}=1/3$. 
From Section \ref{Spectrum-quasiperiodic} we have that the limit spectrum $\sigma(A^\alpha(0)\rho)$ is given by the reciprocal of the Dirichlet eigenvalues  of
\begin{equation}\label{dirichletA}
\mathcal{L}_Du=\rho^1\hat{\eta} u \hbox{ for $\hat{\eta}>0$ and $u\in H^1_0(D)^3$}. 
\end{equation}
We set $\eta_j=\hat{\eta}^{-1}_j$, and write this as $\sigma(A^\alpha(0)\rho)=\{\eta_j\}_{j=1}^\infty$.
Let $\eta_j\in\sigma(A^\alpha(0)\rho)$ and let  $\tilde{\eta}$ be the minimizer of $\min_{j,j'\in\mathbb{N}}|\eta_j-\eta_{j'}|$.
Then $d=\frac{1}{2}|\eta_j-\tilde{\eta}|$.
Substitution of $d$ and $z^*$ into \eqref{R3} shows that the radius of convergence $r^*$ is defined explicitly in terms of the physical geometry of the inclusions and the Dirichlet spectra of the inclusions and is given by 
\begin{equation}\label{rad1}
    r^*=\frac{\mu_1|\alpha|^2|\eta_j-\tilde{\eta}|}{8\rho^{'}+3\mu_1|\alpha|^2|\eta_j-\tilde{\eta}|}\; \text{for}\;\alpha\neq0.
\end{equation}
From Section \ref{Spectrum-periodic}, theorem \ref{SSt6.1} the spectra is discrete and $\sigma(A^0(0)\rho)=\{\delta'^{-1}_j\}_{j\in\mathbb{N}}\cup\{\nu^{-1}_j\}_{j\in\mathbb{N}}$. We denote generic elements of this discrete spectra as $\gamma_j$ and for a fixed element $\gamma_j$ we let $\tilde{\gamma}$ be the minimizer of $\min_{j,j'\in\mathbb{N}}|\gamma_j-\gamma_{j'}|$.
Then $d=\frac{1}{2}|\gamma_j-\tilde{\gamma}|$.
Substitution of $d$ and $z^*$ into \eqref{R7} shows that the  radius of convergence $r^*$ is defined explicitly in terms of the physical geometry of the inclusions and the Dirichlet spectra of the inclusions and is given by 
\begin{equation}\label{rad2}
    r^*=\frac{\pi^2\mu_1|\gamma_j-\tilde{\gamma}|}{2\rho^{'}+3\pi^2\mu_1|\gamma_j-\tilde{\gamma}|}\; \text{for}\;\alpha=0.
\end{equation}

The convergence rates are given by Theorem \ref{Rt4} using.the values of $r^*$ given by \eqref{rad1} and \eqref{rad2}
\section{Emergence of Bandgaps in the high contrast limit}
\label{sec-bandgap}

In this section we identify conditions which are sufficient for the  emergence of band gaps for sufficiently large contrast. Here the contrast $k$ is real and taken to be in the interval $1\leq k<\infty$. We order the Dirichlet spectrum of the inclusion $D$ by minmax and $0<\delta_1\leq\delta_2\leq\cdots\delta_j\leq\delta_{j+1}\cdots$, $\delta_j\rightarrow\infty$ as $j\rightarrow \infty$. Let $\{\delta'_j\}_{j=1}^{\infty}$ be the Dirichlet eigenvalues associated with zero average eigenfunctions and let $\{\delta^*_j\}_{j=1}^{\infty}$ be the Dirichlet eigenvalues for which there are non-zero average eigenfunctions.
Recall the limit spectra defined in Theorem \ref{SSt6.1}
\begin{equation}
    \label{neumann}
    \sigma(A^0(0)\rho)=\{\delta'^{-1}_j\}_{j=1}^\infty\cup\{\nu^{-1}_j\}_{j=1}^\infty,
\end{equation}
and  $\{\nu_j\}_{j\in\mathbb{N}}$ are the positive roots of \eqref{SS3}. The eigenspaces of elements of $\sigma(A^0(0)\rho)$ are  orthogonal to the $3$ dimensional space of rigid translations. Let  $\omega_j=\lim_{k\rightarrow\infty}\xi_{j+3}(k,0)$ and from Theorem \ref{SSt6.1} and \eqref{alt} one has the alternative  $\omega_j=\delta'_j$ or $\omega_j=\nu_j$. In what follows we assume that the inclusions are in the class $P_\theta$.
We have the interlacing theorem.
\begin{theorem}\label{Interlacing}
\begin{equation}\label{1stinterlace}
\delta_j\leq{\omega}_j\leq\delta_{j+3}.
\end{equation}
\end{theorem}
One can characterize pass bands in the high contrast limit. Set
\begin{equation}\label{defpass}
a_j(k)=min_{\alpha\in Y^*}\xi_j(k,\alpha),\qquad b_j(k)=max_{\alpha\in Y^*}\xi_{j}(k,\alpha),
\end{equation}

The pass band for $k<\infty$ is given by $[a_j(k),b_j(k)]$ and
\begin{theorem}\label{Interval}
\begin{equation}\label{Band}
\lim_{k\rightarrow\infty}[a_j(k),b_j(k)]=[\omega_j,\delta_{j+3}],
\end{equation}
here $\lim_{k\rightarrow\infty}[a_j(k),b_j(k)]$ can reduce to the single point $\omega_j=\delta_{j+3}$,
and the band structure in the high contrast limit $k\rightarrow\infty$ is given by
\begin{equation}\label{Bands}
[0,\delta_1]\cup[0,\delta_2]\cup[0,\delta_3]\bigcup_{j\geq 1}[\omega_j,\delta_{j+3}].
\end{equation}
\end{theorem}
\noindent A criterion for band gap opening now follows.
\begin{theorem}[\bf{Criterion for band gap opening}]
\label{bandgappp}
A band gap exists if
\begin{equation}\label{Bandgap}
\delta_{j+2}<{\omega}_j.
\end{equation}
\end{theorem}
Lemma \ref{minmax} together with Theorem \ref{bandgappp} shows that the band gap increases with decreasing $\rho_2$, this is consistent with the two dimensional result of \cite{AmmariKangLee} and the experimental findings of \cite{EconomouSigalas}.
Theorem \ref{Interlacing}, Theorem \ref{Interval} and the band gap criterion are established in \cite{AmmariKangLee} and \cite{AmmariKangLee1} for the two dimensional case. The variational arguments of \cite{AmmariKangLee} together with Lemma \ref{minmax} can be readily used to prove Theorem \ref{Interlacing} in the three dimensional case. We now provide the proof of Theorem \ref{Interval} in three dimensions using the Lipschitz continuity of the eigenvalues $\xi_j(k,\alpha)$ about $\alpha=0$ for fixed $k$. Using the minmax formulation of eigenvalues one deduces as in \cite{conca} that
\begin{equation}
    \label{lip}
    |\xi_{j+3}(k,0)-\xi_{j+3}(k,\alpha)|<Ck|\alpha|,
\end{equation}
where $C$ is independent of $\alpha$ and $k$. Equation \eqref{R11} of Theorem \ref{Rt4} gives the convergence
\begin{equation}\label{nuconvg}
|\omega_j-\xi_{j+3}(k,0)|<C\frac{d}{k(r^*-(1/k))}.
\end{equation}
So given any $\epsilon>0$ we can find a pair $(k',\alpha')\in\mathbb{R}^+\times Y^\ast$ such that
\begin{equation}
    \label{tight}
    |\omega_j-\xi_{j+3}(k',\alpha')|<\epsilon.
\end{equation}
On the other hand from the minmax formulation $\xi_{j+3}(k,\alpha)$ is monotone increasing with $k$ and for $\alpha\not=0$,  we have  from Theorem \ref{Rt4} that $\xi_{j+3}(k,\alpha)\rightarrow\delta_{j+3}$ as $k\rightarrow \infty$. So Theorem \ref{Interval} follows immediately from \eqref{tight} and these observations.

In closing we show that the symmetry of inclusion domains $D$ provides a  new condition on the interlacing of $\{\nu_j\}_{j=1}^\infty$ and $\{\delta_j^*\}_{j=1}^{\infty}$.
\begin{theorem}\label{Interlacing2}
Suppose $D$ is invariant under the cubic group of rotations then  we have the interlacing
\begin{equation}\label{3rdinterlace}
\nu_{j-1}<\delta^\ast_j<\nu_{j}.
\end{equation}
\end{theorem}
\begin{proof}
Since $D$ is invariant under the cubic group of rotations $M(\nu)=\lambda(\nu)I$, where $I$ is the $3\times 3$ identity. Here  $\lambda(\nu)$ is a real valued function of $\nu$ and $\det\left\{M(\nu)\right\}=\lambda^3(\nu)$ so $\nu_j$ are the roots of the equation $\lambda(\nu)=0$. For any constant vector $\vec{v}$ in $\mathbb{R}^3$ we have
\begin{equation}
    \label{merimorphic}
    \lambda({\nu})= \frac{M(\nu)\vec{v}\cdot\vec{v}}{|\vec{v}|^2}=\int_Y\rho(x)\;dx-\nu\sum_{j\in\mathbb{N}}\frac{a^2_j}{\nu-\delta^*_j},
\end{equation}
where $a^2_j={|\int_D\rho^1 {\psi}_j\;dx\cdot\vec{v}|^2}/{|\vec{v}|^2}>0$ and $\delta_j^\ast$ are only associated with nonzero mean eigenfunctions. For $\delta^\ast_{j-1}<\nu<\delta^\ast_{j}$, calculation shows $-\infty<\lambda(\nu)<\infty$, with $\lambda'(\nu)>0$. From this we conclude $\delta^\ast_{j}<\nu_j<\delta^\ast_{j+1}$ and we have the interlacing $\nu_{j-1}<\delta^\ast_j<\nu_{j}$. Thus a Dirichlet eigenvalue associated with non zero mean eigenfunctions always lies strictly  between successive roots
$\nu_{j-1}$ and $\nu_j$.
\end{proof}

\section{Layer potential representation of operators in power series}
\label{sec-layerpotential}
In this section we identify explicit formulas for the operators $A^\alpha_n$ appearing in the power series \eqref{P-11}. It is shown that $A^\alpha_n, n\neq 0$ can be expressed in terms of integral operators associated with layer potentials and we establish Theorem \ref{Rt3}.\\
Recall that $A^\alpha(z)-A^\alpha(0)$ is given by
\begin{equation}
\label{L1}
 ( z P^\alpha_1+\sum_{-\frac{1}{2}<\tau_i(\alpha)<\frac{1}{2}}z[(1/2+\tau_i(\alpha))+z(1/2-\tau_i(\alpha))]^{-1}P_{\tau_i(\alpha)})(-\mathcal{L}_\alpha)^{-1}.  
\end{equation}
Notice that 
\begin{equation}
\label{L2}
 [(1/2+\tau_i(\alpha))+z(1/2-\tau_i(\alpha))]^{-1}=(\tau_i(\alpha)+1/2)^{-1}\sum_{n=0}^{\infty}z^n\Big(\frac{\tau_i(\alpha)-1/2}{\tau_i(\alpha)+1/2}\Big)^n,
\end{equation}
therefore
\begin{equation}
\label{L3}
  A^\alpha(z)-A^\alpha(0)= ( z P^\alpha_1+\sum_{n=1}^{\infty}z^n\sum_{-\frac{1}{2}<\tau_i(\alpha)<\frac{1}{2}}(\tau_i(\alpha)+1/2)^{-1}\Big(\frac{\tau_i(\alpha)-1/2}{\tau_i(\alpha)+1/2}\Big)^{n-1}P_{\tau_i(\alpha)}P^\alpha_3)(-\mathcal{L}_\alpha)^{-1}
\end{equation}
It follows that
\begin{equation}
\label{L4}
   A^\alpha_1= [P^\alpha_1+\sum_{-\frac{1}{2}<\tau_i(\alpha)<\frac{1}{2}}(\tau_i(\alpha)+1/2)^{-1}P_{\tau_i(\alpha)}P^\alpha_3](-\mathcal{L}_\alpha)^{-1}.
\end{equation}
and 
\begin{equation}
\label{L5}
     A^\alpha_n= [\sum_{-\frac{1}{2}<\tau_i(\alpha)<\frac{1}{2}}(\tau_i(\alpha)+1/2)^{-1}\Big(\frac{\tau_i(\alpha)-1/2}{\tau_i(\alpha)+1/2}\Big)^{n-1}P_{\tau_i(\alpha)}P^\alpha_3](-\mathcal{L}_\alpha)^{-1}.
\end{equation}
recall also that we have the resolution of the identity
\begin{equation}
\label{L6}
    I_{H_\#^1(\alpha, Y)^3}=P^\alpha_1+P^\alpha_2+P^\alpha_3\;\text{with}\;P^\alpha_3=\sum_{-\frac{1}{2}<\tau_i(\alpha)<\frac{1}{2}}P_{\tau_i(\alpha)},
\end{equation}
and the spectral representation
\begin{equation}
    \label{L7}
    \begin{aligned}
    \langle Tu,v\rangle&=\langle (\mathcal{S}^\alpha_D(\mathcal{K}^{-\alpha}_{D})^*(\mathcal{S}^\alpha_D)^{-1})P^\alpha_3u+\frac{1}{2}P^\alpha_1u-\frac{1}{2}P^\alpha_2u,v\rangle\\
    &=\langle \sum_{-\frac{1}{2}<\tau_i(\alpha)<\frac{1}{2}}\tau_i(\alpha)P_{\tau_i(\alpha)}u+\frac{1}{2}P^\alpha_1u-\frac{1}{2}P^\alpha_2u,v\rangle.
    \end{aligned}
\end{equation}
Adding $\frac{1}{2}I$ to both sides of the above equation, we obtain
\begin{equation}
\label{L8}
\begin{aligned}
    \langle(Tu+\frac{1}{2}I)u,v\rangle&=\langle (\sum_{-\frac{1}{2}<\tau_i(\alpha)<\frac{1}{2}}(\tau_i(\alpha)+\frac{1}{2})P_{\tau_i(\alpha)}+P^\alpha_1)u,v\rangle\\
    &=\langle ((\mathcal{S}^\alpha_D(\mathcal{K}^{-\alpha}_{D})^*(\mathcal{S}^\alpha_D)^{-1}+\frac{1}{2}P^\alpha_3)P^\alpha_3+P^\alpha_1)u,v\rangle\\
    &=\langle ((\mathcal{S}^\alpha_D(\mathcal{K}^{-\alpha}_{D})^*+\frac{1}{2}\tilde{I})(\mathcal{S}^\alpha_D)^{-1})P^\alpha_3+P^\alpha_1)u,v\rangle,
    \end{aligned}
\end{equation}
where $\tilde{I}$ is the identity on $H^{-1/2}(\partial D)^3$. Now from \eqref{L8} we see that
\begin{equation}
\label{L9}
    \begin{aligned}
    \sum_{-\frac{1}{2}<\tau_i(\alpha)<\frac{1}{2}}(\tau_i(\alpha)+\frac{1}{2})^{-1}P_{\tau_i(\alpha)}P^\alpha_3&=(\mathcal{S}^\alpha_{D}(\mathcal{K}^{-\alpha}_{D})^*(\mathcal{S}^\alpha_{D})^{-1}+\frac{1}{2}P^\alpha_3)^{-1}P^\alpha_3\\
    &=((\mathcal{S}^\alpha_{D}(\mathcal{K}^{-\alpha}_{D})^*+\frac{1}{2}\tilde{I})(\mathcal{S}^\alpha_{D})^{-1})^{-1}P^\alpha_3.
    \end{aligned}
\end{equation}
Using the first line of \eqref{L4} and \eqref{L9}, we obtain
\begin{equation}
\label{L10}
    A^\alpha_1=[\mathcal{S}^\alpha_{D}(\mathcal{K}^{-\alpha}_{D})^*+\frac{1}{2}\tilde{I})(\mathcal{S}^\alpha_{D})^{-1})^{-1}P^\alpha_3+P^\alpha_1](-\mathcal{L}_\alpha)^{-1}.
\end{equation}
For higher order terms, by the mutual orthogonality of the projections $P_{\tau_i(\alpha)}$, we have that
\begin{equation}
\label{L11}
\begin{aligned}
  &\sum_{-\frac{1}{2}<\tau_i(\alpha)<\frac{1}{2}}(\tau_i(\alpha)+1/2)^{-1}\Big(\frac{\tau_i(\alpha)-1/2}{\tau_i(\alpha)+1/2}\Big)^{n-1}P_{\tau_i(\alpha)} \\
  &=\Big(\sum_{-\frac{1}{2}<\tau_i(\alpha)<\frac{1}{2}}(\tau_i(\alpha)+1/2)^{-1}P_{\tau_i(\alpha)}\Big)\Big( \sum_{-\frac{1}{2}<\mu_i<\frac{1}{2}}\Big(\frac{\tau_i(\alpha)-1/2}{\tau_i(\alpha)+1/2}\Big)P_{\tau_i(\alpha)}\Big)^{n-1}\\
  &=\Big(\sum_{-\frac{1}{2}<\tau_i(\alpha)<\frac{1}{2}}(\tau_i(\alpha)+1/2)^{-1}P_{\tau_i(\alpha)}\Big)\Big( \sum_{-\frac{1}{2}<\tau_i(\alpha)<\frac{1}{2}}(\tau_i(\alpha)-1/2)P_{\tau_i(\alpha)}\Big)^{n-1}\Big( \sum_{-\frac{1}{2}<\tau_i(\alpha)<\frac{1}{2}}(\tau_i(\alpha)+1/2)P_{\tau_i(\alpha)}\Big)^{1-n}.
  \end{aligned}
\end{equation}
As above, we have that
\begin{equation}
\label{L12}
    \begin{aligned}
  \Big(\sum_{-\frac{1}{2}<\tau_i(\alpha)<\frac{1}{2}}(\tau_i(\alpha)+1/2)^{-1}P_{\tau_i(\alpha)}\Big)&=  \mathcal{S}^\alpha_{D}((\mathcal{K}^{-\alpha}_{D})^*+\frac{1}{2}\tilde{I})^{-1}(\mathcal{S}^\alpha_{D})^{-1})^{-1}P^\alpha_3,\\
  \Big(\sum_{-\frac{1}{2}<\tau_i(\alpha)<\frac{1}{2}}(\tau_i(\alpha)+1/2)P_{\tau_i(\alpha)}\Big)&=\mathcal{S}^\alpha_{D}((\mathcal{K}^{-\alpha}_{D})^*+\frac{1}{2}\tilde{I})(\mathcal{S}^\alpha_{D})^{-1}P^\alpha_3,\\
  \Big(\sum_{-\frac{1}{2}<\tau_i(\alpha)<\frac{1}{2}}(\tau_i(\alpha)-1/2)P_{\tau_i(\alpha)}\Big)&=\mathcal{S}^\alpha_{D}((\mathcal{K}^{-\alpha}_{D})^*-\frac{1}{2}\tilde{I})(\mathcal{S}^\alpha_{D})^{-1}P^\alpha_3.
    \end{aligned}
\end{equation}
Combining \eqref{L12}, \eqref{L11}, and \eqref{L4} we obtain the layer-potential representation for $A^\alpha_n$,
\begin{equation}
\label{L13}
    A^\alpha_n=\mathcal{S}^\alpha_{D}((\mathcal{K}^{-\alpha}_{D})^*+\frac{1}{2}\tilde{I})^{-1}(\mathcal{S}^\alpha_{D})^{-1})^{-1}[\mathcal{S}^\alpha_{D}((\mathcal{K}^{-\alpha}_{D})^*-\frac{1}{2}\tilde{I})\mathcal{S}^\alpha_{D}((\mathcal{K}^{-\alpha}_{D})^*+\frac{1}{2}\tilde{I})^{-1}(\mathcal{S}^\alpha_{D})^{-1}]^{n-1}P^\alpha_3(-\mathcal{L}_\alpha)^{-1}.
\end{equation}

\section{Explicit first order correction to the Bloch band structure in the high contrast limit}
\label{sec-explicit first order}
In this section we develop explicit formulas for the second term in the power series
\begin{equation}
\label{E1}
  \beta^\alpha_j(z)=\beta^\alpha_j(0)+z\beta^\alpha_{j,1}+z^2\beta^\alpha_{j,2}+\cdots
\end{equation}
for simple eigenvalues. We use analytic representation of $A^\alpha(z)$ and the Cauchy Integral Formula to represent $\beta^\alpha_{j,1}$
\begin{equation}
\label{E2}
    \begin{aligned}
    \beta^\alpha_{j,1}&=\frac{1}{2\pi im}\operatorname{tr}\oint_\Gamma A^\alpha_1\rho R(0,\zeta)\;d\zeta\\
    &=\frac{1}{2\pi im}\operatorname{tr}(A^\alpha_1\rho\oint_\Gamma R(0,\zeta)\;d\zeta)\\
    &=\frac{1}{m}\operatorname{tr}(A^\alpha_1\rho P(0))=\frac{1}{m}\sum_{k=1}^{m}\langle \varphi_k,A^\alpha_1\rho P(0)\varphi_k\rangle_{L^2_\#(\alpha, Y)^3}
    \end{aligned}
\end{equation}
Here $P(0)$ is the $L^2_\#(\alpha, Y)^3$ projection onto the eigenspace corresponding to the Dirichelt eigenvalue $(\beta^\alpha_j(0))^{-1}$ of $-\mathcal{L}_D$. For simple eigenvalue consider the normalized eigenvector $P(0)\varphi=\varphi$ and
\begin{equation}
    \label{E3}
    \beta^\alpha_{j,1}=\langle\varphi,A^\alpha_1\rho P(0)\rangle_{L^2_\#(\alpha, Y)^3}
\end{equation}
We apply the integral operator representation of $A^\alpha_1$ to deliver an explicit formula for the first order term $ \beta^\alpha_{j,1}$ in the series for $ \beta^\alpha_j(z)$. The explicit formula is given by the following theorem.
\begin{theorem}
\label{Et11.1}
Let $\beta^\alpha_j(z)$ be an eigenvalue of $A^\alpha(z)\rho$. Then for $|z|<r^*$ there is a $\beta_j(0)\in\sigma(-\mathcal{L}^{-1}_D)$ with corresponding eigenfunction $\varphi_j$ such that 
\begin{equation}
\label{E4}
\beta^\alpha_j(z)=\beta^\alpha_j(0)+z(\frac{\beta_j(0)}{\rho^1})^2\int_{Y\setminus D}\mathbf{C}^1\mathcal{E}(v):\overline{\mathcal{E}(v)}dx+z^2\beta^\alpha_{j,2}+\cdots
\end{equation}
where $v$ takes $\alpha-$quasi periodic boundary conditions on $\partial Y$, and $\mathcal{L}v=0$ in $Y\setminus D$, and takes the Neumann boundary conditions on $\partial D$ given by
\[
  n\cdot\mathbf{C}^1\mathcal{E}(v)|_{\partial D^+}= n\cdot\mathbf{C}^1\mathcal{E}(\varphi)|_{\partial D^-},
\]
\end{theorem}
\begin{remark}
From Theorem \ref{Rt5}, we have eigenvalues $\xi_j^\alpha(k)=(\beta^\alpha_j(1/k))^{-1}$, for $j\in\mathbb{N}$. The high coupling limit expansion for $\xi_j^\alpha(k)$ is written in terms of the expansion $\beta^\alpha_j(z)=\beta_j(0)+z\beta^\alpha_{j,1}+\cdots$ as 
\begin{equation}
    \label{E5}
    \begin{aligned}
    \xi_j^\alpha(k)&=(\beta_j(0))^{-1}-\frac{1}{k}(\beta_j(0))^{-2}\beta^\alpha_{j,1}+\cdots\\
    &=\xi_j(0)-\frac{1}{k\rho^2_1}\int_{Y\setminus D}\mathbf{C}^1\mathcal{E}(v):\overline{\mathcal{E}(v)}+\cdots,
    \end{aligned}
\end{equation}
where $\xi_j(0)=(\beta_j(0))^{-1}$ is the $j^{th}$ Dirichlet eigenvalue for the homogeneous Lam\'e  operator in $D$. This is consistent with the formula for the leading order terms presented in \cite{AmmariKangLee1} for 2 dimensional elasticity.
\end{remark}
\begin{proof}
Recall from the previous section that 
\begin{equation}
\label{E6}
    \begin{aligned}
    A^\alpha_1&=[\mathcal{S}^\alpha_D((\tilde{\mathcal{K}}^{-\alpha}_D)^*+\frac{1}{2}\tilde{I})^{-1}(\mathcal{S}^\alpha_D)^{-1}P^\alpha_3+P^\alpha_1](-\mathcal{L}_\alpha)^{-1}\\
    &=K^\alpha_1(-\mathcal{L}_\alpha)^{-1},
    \end{aligned}
\end{equation}
where $K^\alpha_1:=\mathcal{S}^\alpha_D((\tilde{\mathcal{K}}^{-\alpha}_D)^*+\frac{1}{2}\tilde{I})^{-1}(\mathcal{S}^\alpha_D)^{-1}P^\alpha_3+P^\alpha_1$. Moreover,
\begin{equation}
    \label{E7}
   (-\mathcal{L}_\alpha)^{-1}f=-\int_Y\mathbf{G}^\alpha(x,y)f(y)dy.
\end{equation}
Since $\varphi$ is a Dirichlet eigenvector of the Lam\'e operator defined on $D$ with eigenvalue $(\beta_j(0))^{-1}$ and $\varphi=0$ in $Y\setminus D$, we have
\begin{equation}
    \label{E8}
    \varphi=-\frac{\beta_j(0)}{\rho^1}\chi_D(-\mathcal{L}\varphi).
\end{equation}
Now from \eqref{E8}
\begin{equation}
    \label{E9}
    \begin{aligned}
    (-\mathcal{L}_\alpha)^{-1}\varphi&=-\frac{\beta_j(0)}{\rho^1}\int_Y \mathbf{G}^\alpha(x,y)\chi_D(\mathcal{L}\varphi)dy\\
    &=-\frac{\beta_j(0)}{\rho^1}\int_D \mathbf{G}^\alpha(x,y)(\mathcal{L}_y\varphi)dy\\
    \end{aligned}
    \end{equation}
 Using integration by parts and adopting index notation where repeated indices indicate summation we get for each component
 \begin{equation*}
    \begin{aligned}
  { [ (-\mathcal{L}_\alpha)^{-1}\varphi]}_i&=-\frac{\beta_j(0)}{\rho^1}(\int_D \partial_j(\mathbf{G}^\alpha(x,y)_{ik}\mathbf{C^1}\mathcal{E}(\varphi)_{kj})dy-\int_D\mathbf{C^1}\mathcal{E}(\varphi)_{kj}:(\mathcal{E}(\mathbf{G}^\alpha(x,y)_i))_{kj}dy)\\
    &=-\frac{\beta_j(0)}{\rho^1}(\mathcal{S}^\alpha_D[\partial_n\varphi|_{\partial D^-}]_{i}(x)-R(x)_i)
    \end{aligned}
\end{equation*}
where the last equality follows from the divergence theorem and definition of the single layer potential $\mathcal{S}^\alpha_D$ and 
\begin{equation}
    \label{E10}
    R(x)_i=\int_D\mathbf{C^1}\mathcal{E}(\varphi)_{kj}:(\mathcal{E}(\mathbf{G}^\alpha(x,y)_i))_{kj}dy).
\end{equation}
Hence
\begin{equation}
    \label{E11}
    [A^\alpha_1\varphi]_i=\frac{\beta_j(0)}{\rho^1}K^\alpha_1(\mathcal{S}^\alpha_D[\partial_n\varphi|_{\partial D^-}]_i(x)-R(x)_i).
\end{equation}
Now we apply the definition of $K^\alpha_1$ and compute $P^\alpha_1R(x)$ and $P^\alpha_3R(x)$. Integrating by parts, we find
\begin{equation}
    \label{E12}
    \begin{aligned}
    R(x)_i&=\int_D\mathbf{C^1}\mathcal{E}(\varphi)_{kj}:(\mathcal{E}(\mathbf{G}^\alpha(x,y)_i))_{kj}dy\\
    &=\int_D(\mathcal{E}(\mathbf{G}^\alpha(x,y)_i))_{kj}:\mathbf{C^1}\mathcal{E}(\varphi)_{kj}dy\\
    &=\int_D\partial_j(\mathbf{C^1}\mathcal{E}(\mathbf{G}^\alpha(x,y)_i)_{kj})\varphi_k))dy-\int_D-\partial_j(\mathbf{C}^1\mathcal{E}(\mathbf{G}^\alpha(x,y)_i)_{kj}\varphi_k dy\\
    &=\varphi(x)_i.
    \end{aligned}
\end{equation}
Thus $P^\alpha_1R(x)=P^\alpha_3R(x)=0$ since $\varphi\in W^\alpha_2$. Then we obtain
\begin{equation}
    \label{E13}
    \begin{aligned}
     \beta^\alpha_{j,1}=\operatorname{tr}(A^\alpha_1\rho P(0))&=\langle\varphi,A^\alpha_1\rho P(0)\rangle_{L^2_\#(\alpha, Y)^3}\\
     &=\langle\varphi,-\frac{\beta_j(0)}{\rho^1}\mathcal{S}^\alpha_D((\tilde{\mathcal{K}}^{-\alpha}_D)^*+\frac{1}{2}\tilde{I})^{-1}[\partial_n\varphi|_{\partial D^-}]\rangle_{L^2_\#(\alpha, Y)^3}
    \end{aligned}
\end{equation}
Let $v\in H_\#^1(\alpha, Y)^3$ be defined 
\begin{equation}
    \label{E14}
    v:=\mathcal{S}^\alpha_D((\tilde{\mathcal{K}}^{-\alpha}_D)^*+\frac{1}{2}\tilde{I})^{-1}[\partial_n\varphi|_{\partial D^-}].
\end{equation}
Then $\mathcal{L}v=0$ in $D$ and $Y\setminus D$, and 
\begin{equation}
    \label{E15}
    n\cdot\mathbf{C}^1\mathcal{E}(v)|_{\partial D^+}= n\cdot\mathbf{C}^1\mathcal{E}(\varphi)|_{\partial D^-}.
\end{equation}
Then we have
\begin{equation}
\begin{aligned}
\label{E16}
   \beta^\alpha_{j,1}&=-\frac{\beta_j(0)}{\rho^1}\langle\varphi,v\rangle=-(\frac{\beta_j(0)}{\rho^1})^2\int_D v(\mathcal{L}\overline{\varphi}) dy\\
   &=-(\frac{\beta_j(0)}{\rho^1})^2(\int_{\partial D}n\cdot\mathbf{C}^1\mathcal{E}(\varphi)|_{\partial D^-}\overline{v}d\sigma-\int_D\mathbf{C}^1\mathcal{E}(\varphi):\overline{\mathcal{E}(v)}dx)\\
   &=-(\frac{\beta_j(0)}{\rho^1})^2(\int_{\partial D}n\cdot\mathbf{C}^1\mathcal{E}(v)|_{\partial D^+}\overline{v}d\sigma-\int_D\mathbf{C}^1\mathcal{E}(\varphi):\overline{\mathcal{E}(v)}dx)
\end{aligned}
\end{equation}
Last an integration by parts yields
\begin{equation}
    \label{E17}
    \begin{aligned}
    \int_D\mathbf{C}^1\mathcal{E}(\varphi):\overline{\mathcal{E}(v)}dx&=\int_D\mathbf{C}^1\overline{\mathcal{E}(v)}:\mathcal{E}(\varphi)dx\\
    &=\int_D\nabla\cdot(\mathbf{C}^1\overline{\mathcal{E}(v)}\varphi)-\mathcal{L}\overline{v}\varphi\\
    &=n\cdot\mathbf{C}^1\overline{\mathcal{E}(v)}|_{\partial D^-}\varphi d\sigma=0.
    \end{aligned}
\end{equation}
Combining this result with the last line of \eqref{E16} and integrating by parts a one time provides a representation of the second term in \eqref{E1}
\begin{equation}
     \beta^\alpha_{j,1}=(\frac{\beta_j(0)}{\rho^1})^2\int_{Y\setminus D}\mathbf{C}^1\mathcal{E}(v):\overline{\mathcal{E}(v)}dx,
\end{equation}
and the theorem follows.
\end{proof}

\section{Derivation of the convergence radius and the separation of spectra}
\label{sec-derivation}
Here we prove Theorems \ref{Rt1} and \ref{Rt2}. To begin, we suppose $\alpha\neq0$ and recall the Neumann series \eqref{P5} and consequently \eqref{P6} and \eqref{P-10} converge provided that 
\begin{equation}
\label{D1}
    \|(A^\alpha(z)\rho-A^\alpha(0)\rho)R(\zeta,0)
    \|_{L[ L^2_\#(\alpha, Y)^3:L^2_\#(\alpha, Y)^3]}<1.
\end{equation}
Then we will compute an explicit upper bound $B(\alpha,z)$ and identify a neighborhood of the origin on the complex plane for which
\begin{equation}
\label{D2}
    \|(A^\alpha(z)\rho-A^\alpha(0)\rho)R(\zeta,0)\|_{L[ L^2_\#(\alpha, Y)^3:L^2_\#(\alpha,Y)^3]}<B(\alpha,z)<1,
\end{equation}
holds for $\zeta\in\Gamma$. The inequality $B(\alpha,z)<1$ will be used first to derive a lower bound on the radius of convergence of the power series expansion of the eigenvalue group about $z=0$. It will then be used to provide a lower bound on the neighborhood of $z=0$ where properties 1 through 3 of Theorem \ref{Rt1} hold.\\
We have the basic statement given by 
\begin{equation}
    \label{D3}
    \begin{aligned}
   & \|(A^\alpha(z)\rho-A^\alpha(0)\rho)R(\zeta,0)\|_{L[ L^2_\#(\alpha, Y)^3:L^2_\#(\alpha, Y)^3]}\leq\\
   &\|(A^\alpha(z)-A^\alpha(0))\|_{L[ L^2_\#(\alpha, Y)^3:L^2_\#(\alpha, Y)^3]}\|\rho\|_{L^\infty(Y)}\|R(\zeta,0)\|_{L[ L^2_\#(\alpha, Y)^3:L^2_\#(\alpha, Y)^3].}
    \end{aligned}
\end{equation}
Here $\zeta\in\Gamma$ as defined in Theorem \ref{Rt1} and elementary arguments deliver the estimate
\begin{equation}
\label{D4}
    \|R(\zeta,0)\|_{L[ L^2_\#(\alpha, Y)^3:L^2_\#(\alpha, Y)^3]}\leq d^{-1},
\end{equation}
where $d$ is given by \eqref{R1}.\\
Next we estimate $\|(A^\alpha(z)-A^\alpha(0))\|_{L[ L^2_\#(\alpha, Y)^3:L^2_\#(\alpha, Y)^3]}$. Denote the energy seminorm of $u$ by
\begin{equation}
    \label{D5}
    \|u\| =\int_{Y}\mathbf{C}^1\mathcal{E} (u):\overline{\mathcal{E}(u)}\;dx.
\end{equation}
First we derive the Poincar\'e inequality between the spaces $L_\#^2(\alpha, Y)^3$ and $H_\#^1(\alpha, Y)^3$ for $\alpha\neq 0$:
\begin{lemma}
\label{Dl1}
\begin{equation}
\label{D6}
\|u\|_{L^2(Y)^3}\leq\frac{\alpha^{-1}}{\sqrt\mu_1}\|u\|.
\end{equation}
\end{lemma}
\begin{proof}
Let $u\in L_\#^2(\alpha, Y)^3$. Set $\xi=2\pi n+\alpha$ and let
\begin{align*}
 -l^{-1}(\xi)v\cdot\overline{v}&=-(-\frac{1}{\mu_1|\xi|^2}(I-\frac{\lambda_1+\mu_1}{\lambda_1+2\mu_1}\frac{(\xi\otimes\xi)}{|\xi|^2}))v\cdot\overline{v}\\
 &=\frac{1}{\mu_1|\xi|^2}(I-\frac{\lambda_1+\mu_1}{\lambda_1+2\mu_1}\frac{(\xi\otimes\xi)}{|\xi|^2}))v\cdot\overline{v},
\end{align*}
then
\begin{align*}
((-\mathcal{L}_\alpha)^{-1}u,u)&=\int_Y((-\mathcal{L}_\alpha)^{-1}u(x)\cdot\overline{u(x)}\;dx\\
&=\int_Y(-\int_Y\mathbf{G}^\alpha(x,y)u(y)\;dy)\cdot\overline{u(x)}\;dx\\
&=\int_Y(-\int_Y(\sum_{n\in\mathbb{Z}^3}l^{-1}(\xi)e^{i\xi\cdot (x-y)})u(y)\;dy)\cdot\overline{u(x)}\;dx\\
&=(\sum_{n\in\mathbb{Z}^3}-l^{-1}(\xi)\widehat{u(\xi)}\cdot{\overline{\widehat{u(\xi)}}}.
\end{align*}
We obtain an upper bound on the quadratic form $-l^{-1}(\xi)\widehat{u(\xi)}\cdot{\overline{\widehat{u(\xi)}}}$. 
Without loss of generality choose the basis $a_1:=\frac{\xi}{|\xi|}=(1,0,0), a_2:=(0,1,0), a_3:=(0,0,1)$ and set $c:=\frac{\lambda_1+\mu_1}{\lambda_1+2\mu_1}$. Then
\begin{align*}
  \frac{1}{\mu_1|\xi|^2}(I-\frac{\lambda_1+\mu_1}{\lambda_1+2\mu_1}\frac{(\xi\otimes\xi)}{|\xi|^2}))&=   \frac{1}{\mu_1|\xi|^2}\left((1-c)a_1\otimes a_1+a_2\otimes a_2+a_3\otimes a_3\right)
\end{align*}
Since $0\leq c\leq 1$, we have $0\leq 1-c \leq 1$. Thus
\begin{align*}
\frac{1}{\mu_1|\xi|^2}(I-\frac{\lambda_1+\mu_1}{\lambda_1+2\mu_1}\frac{(\xi\otimes\xi)}{|\xi|^2}))v\cdot\overline{v}
\leq \frac{1}{\mu_1|\xi|^2}|v|^2.
\end{align*}
Hence we obtain the upper bound
\begin{equation}
\label{D7}
  ((-\mathcal{L}_\alpha)^{-1}u,u)=(\sum_{n\in\mathbb{Z}^3}-l^{-1}(\xi)\widehat{u(\xi)}\cdot{\overline{\widehat{u(\xi)}}}
 \leq\sum_{n\in\mathbb{Z}^3}\frac{1}{\mu_1|\xi|^2}|\widehat{u(\xi)}|^2\end{equation}
 \begin{equation*}
  \begin{aligned}
  &\leq\frac{1}{\mu_1|\alpha|^2}\sum_{n\in\mathbb{Z}^3}|\widehat{u(\xi)}|^2\\
  &=\frac{1}{\mu_1|\alpha|^2}\|u\|^2_{L^2(Y)^3}.
  \end{aligned}
  \end{equation*}
 
 Let $v\in L^2_\#(\alpha, Y)^3$. Then notice that
 \begin{align*}
    \langle (-\mathcal{L}_\alpha)^{-1}v,v\rangle &=\int_{Y}\mathbf{C}^1\mathcal{E} ((-\mathcal{L}_\alpha)^{-1}v):\overline{\mathcal{E}(v)}\;dx\\
    &=\int_{Y}-\nabla\cdot\mathbf{C}^1\mathcal{E} ((-\mathcal{L}_\alpha)^{-1}v)\cdot\overline{v}\;dx\\
    &=\int_{Y}-\mathcal{L}_\alpha ((-\mathcal{L}_\alpha)^{-1}v)\cdot\overline{v}\;dx\\
    &=\int_Y v\cdot\overline{v}\;dx=\|v\|_{L^2(Y)^3}.
 \end{align*}

Now, from the Cauchy inequality we have
\begin{equation}
\label{D8}
\|v\|^2_{L^2(Y)^3}= \langle (-\mathcal{L}_\alpha)^{-1}v,v\rangle\leq\|(-\mathcal{L}_\alpha)^{-1}v\|\|v\|.
\end{equation}
Applying \eqref{D7} we get
\begin{equation}
\label{D9}
\|(-\mathcal{L}_\alpha)^{-1}v\|=\langle (-\mathcal{L}_\alpha)^{-1}v, (-\mathcal{L}_\alpha)^{-1}v\rangle^{1/2}=((-\mathcal{L}_\alpha)^{-1}v,v)^{1/2}\leq\frac{\alpha^{-1}}{\sqrt\mu_1}\|v\|_{L^2(Y)^3}
\end{equation}
and the Poincare inequality follows from \eqref{D8} and \eqref{D9}.
\end{proof}
To obtain the Poincare estimate for $\alpha=0$, we replace $\xi=2\pi n$ in the above proof. Then following will be the corresponding inequality for $\eqref{D7}$.
\begin{equation*}
   ((-\mathcal{L}_0)^{-1}u,u)
 \leq\sum_{n\in\mathbb{Z}^3\setminus \{0\}}\frac{1}{\mu_1|\xi|^2}|\widehat{u(\xi)}|^2  \leq\sum_{n\in\mathbb{Z}^3\setminus \{0\}}\frac{1}{\mu_1|2\pi n|^2}|\widehat{u(\xi)}|^2 
 \end{equation*}
\begin{equation*}
  \begin{aligned}
  &\leq\frac{1}{4\pi^2\mu_1}\sum_{n\in\mathbb{Z}^3\setminus \{0\}}|\widehat{u(\xi)}|^2\\
  &=\frac{1}{4\pi^2\mu_1}\|u\|^2_{L^2(Y)^3}.
  \end{aligned}
\end{equation*}
  Hence for $\alpha=0$, the Poincare inequality becomes
\begin{equation*}
\|v\|_{L^2(Y)^3}\leq\frac{1}{2\pi\sqrt\mu_1}\|v\|.
\end{equation*}
  For any $v\in L^2_\#(\alpha, Y)^3$, we apply \eqref{D6} to find 
  \begin{equation}
      \label{D10}
      \begin{aligned}
      &\|(A^\alpha(z)-A^\alpha(0))v\|_{L^2(Y)^3}\\
      &\leq\frac{|\alpha|^{-1}}{\sqrt{\mu_1}}\|(A^\alpha(z)-A^\alpha(0))v\|\\
      &\leq\frac{|\alpha|^{-1}}{\sqrt{\mu_1}}\|((T_{k}^\alpha)^{-1}-P^\alpha_2)(-\mathcal{L}_\alpha)^{-1}v\|\\
      &\leq\frac{|\alpha|^{-1}}{\sqrt{\mu_1}}\|((T_{k}^\alpha)^{-1}-P^\alpha_2)\|_{{L[ H^1_\#(\alpha, Y)^3:H^1_\#(\alpha, Y)^3]}}\|(-\mathcal{L}_\alpha)^{-1}v\|.
      \end{aligned}
  \end{equation}
  Applying \eqref{D9} and \eqref{D10} delivers the upper bound:
  \begin{equation}\label{D11}
  \begin{aligned}
      &\|(A^\alpha(z)-A^\alpha(0))\|_{L[ L^2_\#(\alpha, Y)^3:L^2_\#(\alpha, Y)^3]}\\
      &\leq \frac{|\alpha|^{-2}}{\mu_1}\|((T_{k}^\alpha)^{-1}-P^\alpha_2)\|_{L[ H^1_\#(\alpha, Y)^3:H^1_\#(\alpha, Y)^3]}.
      \end{aligned}
  \end{equation}
  
  The next step is to obtain an upper bound on $\|((T_{k}^\alpha)^{-1}-P^\alpha_2)\|_{L[ H^1_\#(\alpha, Y)^3:H^1_\#(\alpha, Y)^3]}$. For all $v\in H^1_\#(\alpha, Y)^3$, we have
  \begin{equation}
  \label{D12}
  \frac{\|((T_{k}^\alpha)^{-1}-P^\alpha_2)v\|}{\|v\|}\leq |z|\{w_1+\sum_{i=1}^n \tilde{w}_i|(1/2+\tau_i(\alpha))+z(1/2-\tau_i(\alpha))|^{-2}\}^{1/2}.
\end{equation}
where  $w_1:=\frac{\|P^\alpha_1v\|^2}{\|v\|^2}, \tilde{w}_i:=\frac{\|P_{\tau_i(\alpha)}v\|^2}{\|v\|^2}$, and  $w_1+\sum_{i=1}^n \tilde{w}_i\leq 1$. Thus maximizing the right hand side is equivalent to calculating
\begin{equation}
\label{D13}
\max_{w_1+\sum_{i=1}^n \tilde{w}_i\leq 1} \{w_1+\sum_{i=1}^n \tilde{w}_i|(1/2+\tau_i(\alpha))+z(1/2-\tau_i(\alpha))|^{-2}\}^{1/2}
\end{equation}
\begin{equation*}
\begin{aligned}
&=\sup\{1,|(1/2+\tau_i(\alpha))+z(1/2-\tau_i(\alpha))|^{-2}\}^{1/2}.
\end{aligned}
\end{equation*}
Hence we maximize the function
\begin{equation}
\label{D14}
    f(x)=|\frac{1}{2}+x+z(\frac{1}{2}-x)|^{-2}
\end{equation}
over $x\in[\tau^-(\alpha),\tau^+(\alpha)]$ for $z$ in a neighborhood about the origin. Let $\operatorname{Re}(z)=u$, $\operatorname{Im}(z)=v$ and we write 
\begin{equation}
    \label{D15}
    \begin{aligned}
     f(x)&=|\frac{1}{2}+x+(u+iv)(\frac{1}{2}-x)|^{-2}\\
     &=((\frac{1}{2}+x+u(\frac{1}{2}-x))^2+v^2(\frac{1}{2}-x)^2)^{-1}\\
     &\leq (\frac{1}{2}+x+u(\frac{1}{2}-x))^{-2}=g(\operatorname{Re}(z),x),
    \end{aligned}
\end{equation}
to get the bound
\begin{equation}
\label{D16}
     \|(T_{k}^\alpha)^{-1}-P^\alpha_2\|_{L[ H^1_\#(\alpha, Y)^3:H^1_\#(\alpha, Y)^3]}\leq|z|\sup\{1,\sup_{ x\in[\tau^-(\alpha),\tau^+(\alpha)]}g(u,x)\}^{1/2}\}.
\end{equation}
We now examine the poles of $g(u,x)$ and the sign of its partial derivative $\partial_x g(u,x)$ when $|u|<1.$ If $\operatorname{Re}(z)=u$ is fixed, then $g(u,x)=(\frac{1}{2}+x+u(\frac{1}{2}-x))^{-2}$ has a pole when $(\frac{1}{2}+x)+u(\frac{1}{2}-x)=0.$ For $u$ fixed this occurs when 
\begin{equation}
\label{D17}
    \hat{x}=\hat{x}(u)=\frac{1}{2}\big(\frac{1+u}{u-1}\big).
\end{equation}
On the other hand if, $x$ is fixed, $g$ has a pole at 
\begin{equation}
\label{D18}
    u=\frac{\frac{1}{2}+x}{x-\frac{1}{2}}.
\end{equation}
The sign of $\partial_x g$ is determined by the formula 
\begin{equation}
    \label{D19}
    \partial_x g(u,x)=N/D,
\end{equation}
where $N=-2(1-u)^2x-(1-u^2)$ and $D=((\frac{1}{2}+x+u(\frac{1}{2}-x))^4\geq0.$ Calculation shows that $\partial_x g<0$ for $x>\hat{x}$, i.e. $g$ is decreasing on$(\hat{x},\infty).$ Similarly,  $\partial_x g>0$ for $x<\hat{x}$ and $g$ is increasing on $(-\infty,\hat{x}).$\\
Now we identify all $u=\operatorname{Re(z)}$ for which $\hat{x}=\hat{x}(u)$ satisfies
\begin{equation}
\label{D20}
   \hat{x}<\tau^-(\alpha)<0. 
\end{equation}
For such $u$, the function $g(u,x)$ will  be decreasing on $[\tau^-(\alpha),\tau^+(\alpha)]$, so that $g(u,\tau^-(\alpha))\geq g(u,x)$ for all $x\in[\tau^-(\alpha),\tau^-]$, providing an upper bound for $\eqref{D16}$.
\begin{lemma}
\label{Dl2}
The set $U$ of $u\in\mathbb{R}$ for which $-\frac{1}{2}<\hat{x}(u)<\tau^-(\alpha)<0$ is given by 
\[
U:=[z^*,1]
\]
where 
\[
-\frac{(\mu_1+\lambda_1)}{(3\mu_1+\lambda_1)}\leq z^*:=\frac{\tau^-(\alpha)+\frac{1}{2}}{\tau^-(\alpha)-\frac{1}{2}}<0.
\]
\end{lemma}
\begin{proof}
Noting $\hat{x}=\hat{x}(u)=\frac{1}{2}\big(\frac{1+u}{u-1}\big)$, we invert and write
\begin{equation}
\label{D21}
   u=\frac{\frac{1}{2}+\hat{x}}{\hat{x}-\frac{1}{2}}  
\end{equation}
We now show that 
\begin{equation}
    \label{D22}
  z^*\leq u\leq 1
\end{equation}
for $\hat{x}\leq\tau^-(\alpha).$ Set $h(\hat{x})=\frac{\frac{1}{2}+\hat{x}}{\hat{x}-\frac{1}{2}}.$ Then
\begin{equation}
\label{D23}
h^{'}(\hat{x})=\frac{-1}{(\hat{x}-\frac{1}{2})^2},
\end{equation}
and so $h$ is decreasing on $(-\infty,\frac{1}{2}).$ Since $\tau^-(\alpha)<\frac{1}{2}, h$ attains a minimum over $(-\infty,\tau^-(\alpha)]$ at $x=\tau^-(\alpha).$ Thus $\hat{x}(u)\leq\tau^-(\alpha)$ implies
\begin{equation}
    \label{D24}
    z^*=\frac{\tau^-(\alpha)+\frac{1}{2}}{\tau^-(\alpha)-\frac{1}{2}}\leq u\leq 1
\end{equation}
as desired.
\end{proof}
Combining Lemma \eqref{Dl2} with inequality \eqref{D16}, noting that $-|z|\leq\operatorname{Re}(z)\leq|z|$ and on rearranging terms we obtain the following corollary.
\begin{corollary}
\label{Dc3}
For $|z|<|z^*|$:
\begin{equation}
\label{D25}
\|(A^\alpha(z)-A^\alpha(0))\|_{L[ L^2_\#(\alpha, Y)^3:L^2_\#(\alpha, Y)^3]}\leq \frac{|\alpha|^{-2}}{\mu_1}|z|(-|z|-z^*)^{-1}(\frac{1}{2}-\tau^-(\alpha))^{-1}.
\end{equation}
\end{corollary}
From Corollary \ref{Dc3}, \eqref{D3}, \eqref{D4} we easily seen that
\begin{equation}
    \label{D26}
    \begin{aligned}
   & \|(A^\alpha(z)\rho-A^\alpha(0)\rho)R(\zeta,0)\|_{L[ L^2_\#(\alpha, Y)^3:L^2_\#(\alpha, Y)^3]}\leq\\
   &B(\alpha,z)=\frac{|\alpha|^{-2}}{\mu_1}|z|(-|z|-z^*)^{-1}(\frac{1}{2}-\tau^-(\alpha))^{-1}d^{-1}\|\rho\|_{L^\infty(Y)^3}.
    \end{aligned}
\end{equation}
a straight forward calculation shows that $B(\alpha,z)<1$ for
\begin{equation}
\label{D27}
    |z|<r^*:=\frac{\mu_1|\alpha|^2 d |z^*|}{\frac{\|\rho\|_{L^\infty(Y)^3}}{\frac{1}{2}-\tau^-(\alpha)}+\mu_1|\alpha|^2 d}
\end{equation}
and property 4 of Theorem \ref{Rt1} is established since $r^*<|z^*|$.
Now we establish properties 1 through 3 of Theorem \ref{Rt1}. First note that inspection of \eqref{P5} shows that if \eqref{D1} holds and if $\zeta\in\mathbb{C}$ belong to the resolvent of $A^\alpha(0)\rho$ then it also belongs to the resolvent of $A^\alpha(z)\rho$. Since \eqref{D1}
holds for $\zeta\in \Gamma$ and $|z|<r^*$, property 1 of Theorem \ref{Rt1} follows. Formula \eqref{P6} shows that $P(z)$ is analytic in a neighborhood of $z=0$ determined by the condition that \eqref{D1} holds for $\zeta\in \Gamma$. The set $|z|<r^*$ lies inside the neighborhood and property 2 of Theorem \eqref{Rt1} is proved. The isomorphimsm expressed in property 3 of Theorem \eqref{Rt1} follows directly Lemma \ref{P-10} (\cite{TKato3}, Chapter I \S 4) which is also valid for Banach space.\\
The proof of Theorem \ref{Rt2} proceed along identical lines. To prove Theorem \ref{Rt2}, we need the following Poincare inequality between $L^2_\#(0, Y)^3$ and $H^1_\#(0, Y)^3$.
\begin{lemma}
\label{Dl4}
\begin{equation}
\label{D28}
\|v\|_{L^2(Y)^3}\leq\frac{1}{2\pi\sqrt\mu_1}\|v\|.
\end{equation}
\end{lemma}
This inequality is established using \eqref{D12} and proceeding using the same steps as in the proof of Lemma \ref{Dl1}. Using \eqref{D28} in place of \eqref{D6} we argue as in the proof of Theorem \ref{Rt1}
 to show that
 \begin{equation}
\label{D29}
    \|(A^0(z)\rho-A^0(0)\rho)R(\zeta,0)\|_{L[ L^2_\#(0, Y)^3:L^2_\#(0, Y)^3]}<1
\end{equation}
holds provided $|z|<r^*$, where $r^*$ is given by \eqref{R7}. This establishes Theorem \ref{Rt2}.\\
 The error estimates presented in Theorem \ref{Rt4} are easily recovered from the arguments in(\cite{TKato3} Chapter II, \S3); for completeness we restate them here. We begin with the following application of Cauchy inequalities to the coefficients $\beta_n^\alpha$ of \eqref{P-10} from (\cite{TKato3} Chapter II, \S3, pg 88): 
\begin{equation}
    \label{D30}
    |\beta_n^\alpha|\leq d(r^*)^{-n}.
\end{equation}
It follows immediately that, for $|z|<r^*$,
\begin{equation}
    \label{D31}
    \Big|\hat{\beta}^\alpha(z)-\sum_{n=0}^{p}z^n\beta^\alpha_n\Big|\leq\sum_{n=p+1}^{\infty}|z|^n|\beta^\alpha_n|\leq \frac{d|z|^{p+1}}{(r^*)^p(r^*-|z|)}.
\end{equation}
completing the proof.\\
For completeness we establish the boundedness and compactness of the operator $B^\alpha(k)$.
 \begin{theorem}
 \label{Dt5}
The operator $B^\alpha(k):L^2_\#(\alpha, Y)^3\mapsto H^1_\#(\alpha, Y)^3$ is bounded for $k\notin Z$.
\end{theorem} 
\
We first prove the result for $\alpha\neq 0$. Let $v\in L^2_\#(\alpha, Y)^3$. Then,
\begin{equation}
\label{D32}
\begin{aligned}
\|B^\alpha(k)v\|&=\|(T_{k}^\alpha)^{-1}(-\mathcal{L}_\alpha)^{-1}v\| \\
&\leq\|(T_{k}^\alpha)^{-1}\|_{L[ H^1_\#(\alpha, Y)^3:H^1_\#(\alpha, Y)^3]}\|(-\mathcal{L}_\alpha)^{-1}v\| \\
&\leq\frac{\alpha^{-1}}{\sqrt\mu_1}\|(T_{k}^\alpha)^{-1}\|_{L[ H^1_\#(\alpha, Y)^3:H^1_\#(\alpha, Y)^3]}\|v\|_{L^2(Y)^3},
\end{aligned}
\end{equation}

where the last inequality follows from \eqref{D9}. Now we need to find an upper estimate for $\|(T_{k}^\alpha)^{-1}\|_{L[ H^1_\#(\alpha, Y)^3:H^1_\#(\alpha, Y)^3]}$. Observe that 
\begin{equation}
\label{D33}
    \frac{\|(T_{k}^\alpha)^{-1}v\|}{\|v\|}\leq \{|z|^2w_1+w_2+|z|^2\sum_{i=1}^{\infty}|(1/2+\tau_i(\alpha))+z(1/2-\tau_i(\alpha))|^{-2}\tilde{w}_i\}^{1/2},
\end{equation}

where $w_1:=\frac{\|P_1v\|^2}{\|v\|^2},w_2:=\frac{\|P_2v\|^2}{\|v\|^2}, \tilde{w}_i:=\frac{\|P_{\tau_i(\alpha)}v\|^2}{\|v\|^2}$. Now
 Notice that $w_1+w_2+\sum_{i=1}^{\infty}\tilde{w}_i=1$, and therefore we obtain the upper estimate 
\begin{equation}
\label{D34}
    \frac{\|(T_{k}^\alpha)^{-1}v\|}{\|v\|}\leq M
\end{equation}
where
\begin{equation}
\label{D35}
    M=\max\{1,|z|^2,\sup_{i}\{|(1/2+\tau_i(\alpha))+z(1/2-\tau_i(\alpha))|^{-1}\}\},
\end{equation}
and this completes the proof. For $\alpha=0$ case, the proof is similar.
\begin{remark}
\label{Dr1}
The Poincare inequalities \eqref{D6} and \eqref{D28} together with Theorem \ref{Dt5} show that $B^\alpha(k):L^2_\#(\alpha, Y)^3\mapsto L^2_\#(\alpha, Y)^3$ is a bounded linear operator mapping $L^2_\#(\alpha, Y)^3$ into itself. The compact embedding of $H^1_\#(\alpha, Y)^3$ into $L^2_\#(\alpha, Y)^3$ shows the operator is compact on $L^2_\#(\alpha, Y)^3$.
 \end{remark}
 
 \section*{Acknowledgements}
This research work is supported in part by NSF Grants DMS-1813698 and DMREF-1921707.

\appendix
 \section{Appendix: Proof of Lemma \ref{W2}}
 \label{appendix a}
 \begin{proof}
Let $x\in Y\setminus D.$  Since $\mathcal{E}(u)=0$ there, then  $u$ has to be a rigid motion in $Y\setminus D.$ However since rigid rotations are not periodic, $u$ must be a rigid translation. Thus we write $u=c$ for $x\in Y\setminus D$ where $c$ is a constant vector in $\mathbb{R}^3$. From the continuity of $u$ across the boundary of $D$ we have $u|_{\partial D}=c$ so we can express $u$ as $u=\tilde{u}+c$ for $x\in Y$ where $\tilde{u}\in \tilde{H}^1_0(D)^3$. Then the condition $\int_Y\rho\ u dx=0$ in $Y$ implies that $c=-\langle \rho\rangle^{-1}\int_{D}\rho^1\,\tilde{u}\,dx$, where $\langle\rho\rangle=\int_Y\,\rho\,dx$ and the Lemma follows. 
\end{proof}

\end{document}